\title{Dynamics of iwip automorphisms of free products}
\author{Ioannis Papavasileiou \and Dionysios Syrigos}
\documentclass[11pt]{article}

\usepackage[top=2cm, bottom=4.5cm, left=2.5cm, right=2.5cm]{geometry}

\usepackage{graphicx}
\usepackage{setspace}

\usepackage{ccfonts}
\usepackage[T1]{fontenc}

\usepackage{mathtools}
\usepackage{lipsum}
\usepackage{xfrac}
\usepackage{amsfonts}
\usepackage{cancel}
\usepackage{soul}
\usepackage{graphics}
\usepackage{multicol}
\usepackage{verbatim}
\usepackage{float}
\usepackage{amsthm}
\usepackage{amsmath}
\usepackage{hyperref}
\usepackage{amssymb}
\usepackage[dvipsnames]{xcolor}
\usepackage{color}

\usepackage[
backend=biber,
style=alphabetic,
sorting=nty,
block=par
]{biblatex}
\usepackage[draft]{todonotes}

\usepackage{soul}

\DeclareFieldFormat{pages}{#1}
\DeclareFieldFormat[inbook]{citebookTitle}{\mkbibquote{#1\isdot}}
\DeclareFieldFormat[inbook]{bookTitle}{\mkbibquote{#1\isdot}} 
\DeclareFieldFormat[inbook]{citebookTitle}{#1}
\DeclareFieldFormat[inbook]{bookTitle}{#1} 

\renewbibmacro*{volume+number+eid}{
  \printfield{volume}
  \setunit*{\addnbspace} 
  \printfield{number}
  \setunit{\addcomma\space}
  \printfield{eid}}
\DeclareFieldFormat[article]{number}{\mkbibparens{#1}}

\renewbibmacro{in:}{
  \ifentrytype{article}
    {}
    {\bibstring{in}
     \printunit{\intitlepunct}}}
\hypersetup{
    colorlinks=true,
    linkcolor=blue,
    filecolor=blue,      
    urlcolor=blue,
    pdftitle={SyrigosPapavasileiou},
    pdfpagemode=FullScreen,
    }

\newtheorem{lemma}{Lemma}[section]
\newtheorem{proposition}{Proposition}[section]
\newtheorem{corollary}{Corollary}[section]
\theoremstyle{definition}
\newtheorem{remark}{Remark}[section] 
\newtheorem{definition}{Definition}[section]
 
\newtheorem*{sublemma*}{Lemma}

\usepackage{titlesec}
\usepackage{titlecaps}
\usepackage{fancyhdr}

\newtheorem{theorem}{Theorem}[section]
\newtheorem{innercustomthm}{Theorem}
\newenvironment{mythm}[1]
  {\renewcommand\theinnercustomthm{#1}\innercustomthm}
  {\endinnercustomthm}
\newtheorem{innercustomth}{Corollary}
\newenvironment{mycol}[1]
  {\renewcommand\theinnercustomth{#1}\innercustomth}
  {\endinnercustomth}

\begin{document}
\newpage
\maketitle
 
\begin{abstract}
Let $G$ be a free product and $\mathrm{Out}(G)$ the outer automorphism group of $G$. In this article using the theory of laminations we give a criterion for a subgroup $H$ of $\mathrm{Out}(G)$ to contain a nonabelian free subgroup. We also study the centraliser of an iwip element of $\mathrm{Out}(G)$ and the stabiliser of its associated lamination.
\end{abstract}
\setstretch{1.4}

\section{Introduction}
The study of outer automorphisms of free groups and their dynamics goes back in $1992$ when Bestvina and Handel  \cite{BestvinaHandeltraintracks} introduced the notion of relative train track representatives to explore the dynamics of an outer automorphism. Later, Bestvina, Feighn and Handel \cite{Lam} transferred the notion of laminations from the theory of compact surfaces to free groups in order to examine iwip (irreducible with irreducible powers or fully irreducible) outer automorphisms of free groups. Using laminations as their main tool, they proved \cite{Tits1}, \cite{Tits2} that the group of outer automorphisms, $\mathrm{Out}(F_n)$, of a finitely generated free group $F_n$, satisfies the strong Tits alternative. Recently, the celebrated result of Bestvina, Feighn and Handel has been extended to the case of finitely generated free-by-finite groups \cite{SykPap}.\par
Let $G=G_1\ast\ldots\ast G_k\ast F_p$ be a free product decomposition of a group $G$ and let $\mathcal{O}$ be the relative outer space corresponding to this decomposition, defined in \cite{Outspaceprod}. Let, also, $\mathrm{Out}(G,\mathcal{G})$ be the group of outer automorphisms of $G$ that preserve the set of conjugacy classes of the $G_i$'s. In order to study $\mathrm{Out}(G,\mathcal{G})$, the second author \cite{syrigos2016irreducible} extended the notion of stable laminations for iwip elements to the context of free products, whereas Guirardel and Horbez \cite{Guirardel2017AlgebraicLF} developed the theory of algebraic laminations. Additionally, the Tits altenative for automorphisms of free products has been studied by Horbez in \cite{horbez2014tits}.\par
In this paper, using laminations and the ping-pong lemma we obtain our first main result which is a criterion for a subgroup $H\leq \mathrm{Out}(G,\mathcal{G})$ to contain a nonabelian free subgroup and which extends \autocite[Corollary~3.4.3]{Tits1}. More precisely, we prove the following:
  \begin{mythm}{\ref{main1}}
Let $[\phi]\in\mathrm{Out}(G,\mathcal{G})$ be an outer automorphism of exponential growth and let $\Lambda^+\in \mathcal{L}([\phi])$ and  $\Lambda^-\in \mathcal{L}([\phi]^{-1})$ be two paired
and $[\phi]$-invariant laminations. Suppose that $H$ is a subgroup of $\mathrm{Out}(G,\mathcal{G})$ containing $[\phi]$ and that there
is an element $[\psi]\in H$ of exponential growth such that generic lines of the four laminations $[\psi]^{\pm1}(\Lambda^{\pm})$ are weakly attracted to $\Lambda^+$ under the action of $[\phi]$ and to $\Lambda^-$ under the action of $[\phi]^{-1}$. Then $H$ contains a free subgroup of rank two.  
  \end{mythm}
Given a group $G$ it is of great interest to know under what circumstances (suitable powers of) two elements of $G$ should generate a free subgroup in $G$. For example in \autocite[Proposition~3.7]{Lam} it is proved that sufficiently high powers $[\phi]^n,[\psi]^m$ of two iwip outer automorphisms $[\phi],[\psi]$ of $F_n$  that do not
have common powers generate a free subgroup of $\mathrm{Out}(F_n)$ (see \cite{Ghosh}, \cite{KapovichLustigPingpong}, \cite{ClayPettet}, \cite{Gultepe} for similar results). Using our second main result (Theorem \ref{thm7.2}) we get the following generalization, which may be viewed as a Tits alternative for two-generator subgroups of $\mathrm{Out}(G,\mathcal{G})$. In fact, we show that two iwip outer automorphisms $[\phi]$ and $[\psi]$ are either "independent", i.e. they have some iterates generating a free subgroup of rank two, or they are "strongly related", i.e. some iterates of them differ by an element which virtually fixes a point of $\mathcal{O}$. 
\begin{mythm}{\ref{main2}}
 Let $[\phi],[\psi]$ be two iwip outer automorphisms in $\mathrm{Out}(G,\mathcal{G})$. Then we have the following dichotomy:
    \begin{enumerate}
        \item either $\langle [\phi]^m,[\psi]^n\rangle\cong F_2$ for some $m,n\in\mathbb{N}$, or
        \item $[\phi]^k=[\psi]^l\cdot [\alpha]$ for some $k,l\in\mathbb{N}$ and $[\alpha]\in\mathrm{Out}(G,\mathcal{G})$ such that $[\alpha]$ (virtually) fixes an element of $\mathcal{O}$.
    \end{enumerate}
\end{mythm}

In the case of free products $G=G_1\ast\ldots\ast G_k\ast F_p$ where each factor $G_i$ is finite, the stabilisers of elements of $\mathcal{O}$ are finite as well and $\mathrm{Out}(G,\mathcal{G})$ is virtually torsion free. Thus, as a corollary, we obtain the following extension of \autocite[Proposition~3.7]{Lam}.
\begin{mycol}{\ref{maincol}}
 Let $[\phi],[\psi]$ be two iwip outer automorphisms of $G=G_1\ast\ldots\ast G_k\ast F_p$, where each factor $G_i$ is finite, that do not
have common powers. Then high powers of $[\phi]$ and $[\psi]$ freely generate $F_2$.     
\end{mycol}
  
  Bestvina, Feighn and Handel \autocite[Theorem~2.14]{Lam} proved that the stabiliser $\mathrm{Stab}(\Lambda)$ of the stable lamination $\Lambda$ of an iwip outer automorphism of $F_n$, which is the stabiliser of $\Lambda$ under the action of $\mathrm{Out}(F_n)$ on the set of stable laminations, is virtually (infinite)
cyclic. In the free product case  this is not always true (see Remark \ref{gibcent}). However, the
following theorem generalises the free group statement, since it is well known that outer automorphisms of $F_n$ fixing points of the outer space $CV_n$ are exactly the outer automorphisms of finite order. 
  \begin{mythm}{\ref{thm7.2}}
    Let $[\phi]\in\mathrm{Out}(G,\mathcal{G})$ be an iwip automorphism and $\mathrm{Stab}(\Lambda)$ the  stabiliser of the associated stable lamination. Then $\mathrm{Stab}(\Lambda)$ is an infinite cyclic extension of a virtually
elliptic subgroup $A$. In other words, for every element $[\psi]$ of $A$ there is a positive integer $k$ and some element $S$ of $\mathcal{O}$ such that $[\psi]^k(S)=S$.
  \end{mythm}
 For a group $G$ and an element $g\in G$, it is also natural to study the centraliser $C(g)$ of $g$ in $G$. In several classes of groups, centralisers of elements are reasonably well-understood
and very useful in the study of the group. For example, Feighn and Handel in \cite{Feighn2006AbelianSO} classified abelian subgroups of $\mathrm{Out}(F_n)$ by studying centralisers of individual elements.
Moreover, a well known result for iwip automorphisms of $F_n$ states that their centralisers
in $\mathrm{Out}(F_n)$, are virtually cyclic (see \cite{Lam} for the original proof, \cite{LustigConjugacy} and \cite{KapovichLustigStabilizers} for alternative
proofs).  In the case of a free product, it is not always true that the centraliser is virtually
cyclic. However, we obtain the following generalisation, as a corollary of the Theorem \ref{thm7.2}.
\begin{mythm}{\ref{thm7.6}}
 Let $[\phi]\in\mathrm{Out}(G,\mathcal{G})$ be an iwip outer automorphism of $G$. Then there is a normal virtually elliptic
subgroup $B$ of $C([\phi])$ such that $C([\phi])$ is a cyclic extension of $B$.
\end{mythm}
The conclusion of the above theorem can be strengthened by proving that there is a point which is (virtually) fixed by all elements of the subgroup $B$.  More specifically,
we prove that every element of the subgroup $B$ (virtually) fixes the set of points $M_{[\phi]}$ of $\mathcal{O}$ for which
our iwip automorphism $[\phi]$  admits a train track representative.  By combining Theorem \ref{thm7.6} with the main result of \cite{francaviglia2020minimally} and some standard properties of the Lipschitz metric of $\mathcal{O}$ we get the following theorem.
\begin{mythm}{\ref{thm7.8}}
    Let $[\phi]\in\mathrm{Out}(G,\mathcal{G})$ be an iwip outer automorhism. Then $C([\phi])$ has a
normal subgroup $B'$ such that $C([\phi])/B'$ is isomorphic to $\mathbb{Z}$. Moreover, for every $T\in M_{[\phi]}$ and for every $[\psi]\in B'$, there is some $k\in\mathbb{Z}$ such that $[\psi]^k(T)=T$.
\end{mythm}
Here is the organization of the paper. After some preliminaries given in section \ref{prelim}, in section \ref{mainthmsection3} we prove our first main theorem, Theorem \ref{main1}, which is a criterion for a subgroup $H$ of $\mathrm{Out}(G,\mathcal{G})$ to contain a free subgroup of rank two. Along the way, we extend certain useful results of \cite{Tits1} to the case of free products. Proposition \ref{prp2.6bestvina}, presented in section \ref{subgroupssection4}, is a key step in proving our second main theorem, Theorem \ref{thm7.2}. In section \ref{stretchingsection5} we give the definition of the stretching homomorphism $\sigma:\mathrm{Stab}(\Lambda)\to\mathbb{Z}$ and we investigate its kernel which plays an important role in the proof of our second main theorem. Section \ref{secondmainsection7} is devoted to the proof of Theorem \ref{thm7.2} concerning the stabiliser $\mathrm{Stab}(\Lambda)$. Furthermore, we establish results regarding $\mathrm{Stab}(\Lambda)$ in some interesting special cases of free products. We start section \ref{appsection8} with some applications on the relative centriliser $C([\phi])$ of an iwip outer automorphism $[\phi]$ and lastly, we obtain the above mentioned strong dichotomy results, Theorem \ref{main2} and Corollary \ref{maincol}, related to the subgroup generated by sufficiently high powers of two iwip outer automorphisms.  

\section{Preliminaries}\label{prelim}

\subsection{Outer space and automorphisms}
Let $G=G_1\ast\cdots\ast G_k\ast F_p$ be a free product decomposition (not necessarily the Grushko one) of a finitely generated group $G$, where $F_p$ is the free group on $\{x_1,\ldots,x_p\}$. The pair $\mathcal{G}=(\{[G_1],\ldots,[G_k]\},p)$, where $[G_i]$ denotes the conjugacy class of $G_i$, is called a \textbf{free factor system} of $G$. A $G$-tree $T$ is called a \textbf{$\mathcal{G}$-tree} if for each $i\in \{1,\ldots,k\}$ there is exactly one orbit of vertices with stabilizers
conjugate to $G_i$ (all other vertices have trivial stabilisers and we refer to them as \textbf{free vertices}) and  all edge stabilisers are trivial. The \textbf{relative outer space} corresponding to $\mathcal{G}$, denoted by $\mathcal{O}=\mathcal{O}(\mathcal{G})$, is defined to be the set of equivalence classes of minimal,
simplicial, metric $\mathcal{G}$-trees with no redundant vertices where the equivalent relation is given by $G$-equivariant isometries (we refer to \cite{Outspaceprod} or \cite{francaviglia2015stretching} for a detailed definition). In order to refer to a point of $\mathcal{O}$ we shall write $T\in \mathcal{O}$ instead of its class $[T]\in \mathcal{O}$. It is well known that $\mathcal{O}$ can be embedded in a Euclidean space $\mathbb{R}^n$ and thus inherits its topology. In particular, every $T\in \mathcal{O}$ determines an (open) simplex $\Delta(T)\subset\mathcal{O}$ which is the set of points of $\mathcal{O}$ obtained from $T$ by changing the lengths of (orbits of) edges in such
a way so that every edge has positive length and the sum of lengths of edges in the quotient graph $T/G$ is equal to $1$. \par
We denote by $\mathrm{Aut}(G,\mathcal{G})$ the group of automorphisms of $G$ that preserve the conjugacy classes of the free factors $G_i$, $i=1,\ldots,k$, i.e. $\phi\in\mathrm{Aut}(G,\mathcal{G})$ if for every $i\in \{1,\ldots,k\}$ there is some $j\in \{1,\ldots,k\}$ such that $\phi([G_i])=[G_j]$. It is clear that the group of inner automorphisms $\mathrm{Inn}(G)$ is a normal subgroup of $\mathrm{Aut}(G,\mathcal{G})$, hence we can form the quotient $\mathrm{Out}(G,\mathcal{G})= \mathrm{Aut}(G,\mathcal{G})/\mathrm{Inn}(G)$. Note that in case where $\mathcal{G}$ corresponds to the Grushko decomposition of $G$, then $\mathrm{Out}(G,\mathcal{G})=\mathrm{Out}(G)$.
\begin{definition}
    Let $T,T'\in\mathcal{O}$. A map $\Tilde{f}:T\to T'$ is called \textbf{$\mathcal{O}$-map} if it is  Lipschitz continuous and $G$-equivariant. An $\mathcal{O}$-map that permutes the orbits of edges is called \textbf{$\mathcal{O}$-permutation}.
\end{definition}

\begin{definition}
   Let $T\in\mathcal{O}$ and $p,q$ two segments of $T$ with the same initial vertex. If $p,q^{-1}$ belong
to distinct orbits, then we say that we have the \textbf{fold} (or folding) $\tau:T\to T'$ that is the map that $G$-equivariantly identifies $p$ with $q$. If it happens to be $T'\in\mathcal{O}$, then we say that $\tau$ is an \textbf{$\mathcal{O}$-fold} (or $\mathcal{O}$-folding). 
\end{definition}

 \begin{definition}
      Let $T,T'\in\mathcal{O}$. An $\mathcal{O}$-map $\Tilde{f}:T\to T'$ is called \textbf{PL}, if it is piecewise linear.
 \end{definition}
From now on we will assume that every PL map is linear on edges.

\subsection{Topological representatives}

   Let $T$ be a tree in $\mathcal{O}$. A map $\Tilde{f}:T\to T$ is called a \textbf{topological representative} of $\phi\in \mathrm{Aut}(G,\mathcal{G})$ if it is $\phi$-equivariant, i.e. $\Tilde{f}(g\cdot \widetilde{x})=\phi(g)\cdot \Tilde{f}(\widetilde{x})$ for every $g\in G,\ \widetilde{x}\in T$. A topological representative of $[\phi]\in\mathrm{Out}(G,\mathcal{G})$ is a topological representative of some $\psi\in [\phi]$. A \textbf{filtration} for the topological representative $\Tilde{f}$ is an increasing sequence
    \[
    \emptyset =T_0\subseteq T_1\subseteq \ldots \subseteq T_m=T
    \]
    of $\Tilde{f}$-invariant $G$-subgraphs
of the tree $T$. The $G$-set of the edges $T_i\setminus T_{i-1}$ is called the \textbf{$i$-stratum} and is denoted by $H_i$. The \textbf{transition matrix} $M(\Tilde{f})$ associated to $\Tilde{f}$ is the $n\times n$ matrix whose $(i,j)$-entry is the number of times the path $\Tilde{f}(\widetilde{e}_i)$ crosses the
orbit $G\cdot \widetilde{e}_j$ of $\widetilde{e}_j$, regardless of orientation. The inner automorphism $\tau_g\in\mathrm{Inn}(G)$, given by $\tau_g(x)=gxg^{-1},\ x\in G$, is represented by the isometry $i_g:t\mapsto gt$. Therefore, if $\phi'=\tau_g\circ \phi$ is another automorphism in the outer class $[\phi]$ of $\phi$, then the composition $\Tilde{f}'=i_g\circ \Tilde{f}$ represents $\phi'$ and $\Tilde{f}$, $\Tilde{f}'$ have the same transition matrix. Thus, the
choice of a particular automorphism $\phi'\in [\phi]$ will have no effect on what
follows. We also write $M_i(\Tilde{f})$ for the submatrix of $M(\Tilde{f})$ corresponding to the stratum $H_i$. If $M_i(\Tilde{f})$ is an irreducible matrix, then the stratum $H_i$ is called \textbf{irreducible}. Every irreducible matrix $M_i(\Tilde{f})$ has a \textbf{Perron–Frobenius eigenvalue} $\lambda_i\ge1$. In the case where $\lambda_i>1$, we say that $H_i$ is an \textbf{exponentially growing stratum}. Otherwise, $M_i(\Tilde{f})$ is a transitive permutation matrix. A nonnegative integral matrix $M$ is called \textbf{aperiodic} if there is some positive power $k$ such that $M^k$  has all its entries positive. If $H_r$ is an exponentially growing stratum, then we say that $H_r$ is \textbf{aperiodic} if its transition matrix is aperiodic. The topological representative $\Tilde{f}$ is called \textbf{eg-aperiodic} if each exponentially growing stratum is aperiodic. For a $G$-invariant subgraph $W\subseteq T$ we denote by $\mathcal{F}(W)$, or $\mathcal{F}(W/G)$, the set which consists of the conjugacy classes of $\pi_1(W_i/G)$ where the $W_i$'s are the connected components of $W$. A filtration as above will be called \textbf{reduced} with respect to $[\phi]$ if whenever a free factor system $\mathcal{G}'$ is $[\phi]^k$-invariant for some $k>0$ and $\mathcal{F}(T_i)\leq \mathcal{G}' \leq \mathcal{F}(T_{i+1})$ (where $\leq$ is the usual partial order on free factor systems),  then either $\mathcal{G}'=\mathcal{F}(T_i)$ or $\mathcal{G}'=\mathcal{F}(T_{i+1})$.

Irreducible automorphisms play a central role in the study of automorphisms of finitely
generated free groups. The notion of (relative) irreducibility can be defined in two equivalent ways. The first definition uses topological representatives.

\begin{definition}
    Let $[\phi]\in\mathrm{Out}(G,\mathcal{G})$ be an outer automorphism of $G=G_1\ast\ldots\ast G_k\ast F_p$ and $\Tilde{f}:T\to T$ a topological representative of $[\phi]$. We call $\Tilde{f}$ \textbf{irreducible} if for every proper, $\Tilde{f}$-invariant, $G$-invariant subgraph $W$ of $T$, the quotient graph $W/G$ is a collection of
subtrees each of which contains at most one non-free vertex. We say that $[\phi]$ is \textbf{irreducible} if every representative $\Tilde{f}:T\to T$ of $[\phi]$ is irreducible. An outer automorphism $[\phi]$ (resp. topological representative $\Tilde{f}$) is called \textbf{reducible} if it is not irreducible. An outer automorphism $[\phi]$ (resp. topological representative $\Tilde{f}$) is called \textbf{irreducible with irreducible powers} (iwip) or fully irreducible if for every positive integer $k$ the outer automorphism $[\phi]^k$ (resp. the topological representative $\Tilde{f}^k$) is irreducible.
\end{definition}

The second, equivalent, definition uses free factor systems (see \cite{francaviglia2015stretching} for a detailed exposition).
\begin{definition}
    Let $[\phi]$ be an outer automorphism of $G$ and $\mathcal{G}=(\{[G_1],\ldots,[G_k]\},r)$  a $[\phi]$-invariant free factor system of $G$, i.e. for every $i\in \{1,\ldots,k\}$ there is some $j\in \{1,\ldots,k\}$ such that $[\phi]([G_i])=[G_j]$. We say that $[\phi]$ is \textbf{irreducible}  relative to $\mathcal{G}$ if $\mathcal{G}$ is a maximal, proper, $[\phi]$-invariant free factor system.
\end{definition}
\begin{definition}
    A \textbf{pre-train track structure} on a $G$-tree $T$ is a $G$-invariant equivalence  relation on the set of germs of edges at each vertex of $T$. Equivalence
classes of germs, are called \textbf{gates}. A \textbf{train track} structure on a $G$-tree $T$ is  a pre-train track structure with at least
two gates at every vertex. A \textbf{turn} is a pair (of germs) of edges emanating at the same vertex. A \textbf{legal turn}
is a turn for which the two germs belong to different equivalent classes. A \textbf{legal path}, is a path that contains only legal turns.
\end{definition}

\begin{definition}
 A topological representative $\Tilde{f}:T\to T$ of $[\phi]$ is called a \textbf{train track representative} of $[\phi]$ if it is a PL-map and there is a  train track
structure on $T$ such that:
\begin{itemize}
    \item $\Tilde{f}$ maps edges to legal paths (in particular, $\Tilde{f}$ does not collapse edges), and
    \item if $\Tilde{f}(\Tilde{v})$  is a vertex, then $\Tilde{f}$ maps inequivalent germs at $\Tilde{v}$ to inequivalent germs at $\Tilde{f}(\Tilde{v})$.
\end{itemize}
\end{definition}

For a tree $T$ and $x,y\in T$ we will denote by $[x,y]$ the geodesic in $T$ from $x$ to $y$.  For a path $p$ in $T$, we will denote by $[p]$ the reduced path with the same
endpoints as $p$. \par
Let $T,T'\in\mathcal{O}$ and $\Tilde{f}:T\to T'$ an $\mathcal{O}$-map. The \textbf{bounded cancellation constant} of $\Tilde{f}$, $\mathrm{BCC}(\Tilde{f})$, is defined to be the supremum of all real numbers $\ell$ with the following property: there are three points $x,y,z\in T$, where $y$ is in the interior of the
unique reduced path $[x,z]$  such that $d_{T'}(\Tilde{f}(y),[\Tilde{f}(x),\Tilde{f}(z)])=\ell$. It is well known that $\mathrm{BCC}(\Tilde{f})$ is finite and depends only on $\Tilde{f}$.

\begin{lemma}[\cite{Goldstein},\cite{Cooper}]
    Let $T,T\in \mathcal{O}$ and $\Tilde{f}:T\to T'$ be a  Lipschitz map. Then
    \[
    \mathrm{BCC}(\Tilde{f})\leq \mathrm{Lip}(\Tilde{f})\cdot \mathrm{vol}(T/G)
    \]
    where $\mathrm{vol}(T/G)$ denotes the sum of lengths of edges of the quotient graph $T/G$.
\end{lemma}
The following lemma shows that the (reduced) images of a given path,
under all $\mathcal{O}$-maps coincide up to a bounded error.
\begin{lemma}[{{\cite[Lemma 2.20]{syrigos2016irreducible}}}]\label{lem2.20}
    Let $\Tilde{f},\Tilde{h}:T\to T'$ be two $\mathcal{O}$-maps. There is a positive constant $C$ (which depends only on $\Tilde{f}, \Tilde{h}$ and $T$)  such that for every path $p$ in $T$, the reduced paths $[\Tilde{f}(p)]$ and $[\Tilde{h}(p)]$ are equal, except possibly for some subpaths near their endpoints whose
lengths are bounded by $C$.
\end{lemma}
\begin{proof}
 We first suppose that there is at least one non-free vertex $\Tilde{v}\in V(T)$. The $G$-equivariance of $\mathcal{O}$-maps implies that $\Tilde{f}(\Tilde{v})=\Tilde{h}(\Tilde{v})$. Let $\Tilde{x},\Tilde{y}\in V(T)$ be two vertices of $T$ and $p=[\Tilde{x},\Tilde{y}]$ the geodesic from $\Tilde{x}$ to $\Tilde{y}$. There are $g_1,g_2\in G$ such that $[\Tilde{x},\Tilde{y}]\subseteq [g_1\Tilde{v},g_2\Tilde{v}]$ and both distances $d(\Tilde{x},g_1\Tilde{v})$, $d(\Tilde{y},g_2\Tilde{v})$ are less than or equal to the
quotient volume $\mathrm{vol}(T/G)$. Indeed, let $T_{\Tilde{x}},T_{\Tilde{y}}$ be the two trees obtained by subtracting $\Tilde{x},\Tilde{y}$ from $T$ along with the connected component that contains the midpoint of $[\Tilde{x},\Tilde{y}]$. If we consider the closed ball $B(\Tilde{x},\mathrm{vol}(T/G))\subseteq T_{\Tilde{x}}$, resp. $B(\Tilde{y},\mathrm{vol}(T/G))\subseteq T_{\Tilde{y}}$, then there is $g_1\in G$, resp. $g_2\in G$, such that $g_1\Tilde{v}\in B(\Tilde{x},\mathrm{vol}(T/G))$, resp. $g_2\Tilde{v}\in B(\Tilde{y},\mathrm{vol}(T/G))$. It is clear that $[\Tilde{x},\Tilde{y}]\subseteq [g_1\Tilde{v},g_2\Tilde{v}]$.\par
Therefore, the reduced image $[\Tilde{f}(p)]$ is contained in the reduced path $[\Tilde{f}(g_1\Tilde{v}),\Tilde{f}(g_2\Tilde{v})]$, except possibly two segments (one near each endpoint) whose lengths
are uniformly bounded by the constant $C_1=\mathrm{vol}(T/G)\mathrm{Lip}(\Tilde{f})$. Similarly, $[\Tilde{f}(p)]$ is contained in the reduced path $[\Tilde{h}(g_1\Tilde{v}),\Tilde{h}(g_2\Tilde{v})]$, except possibly two segments (one near each endpoint) whose lengths
are uniformly bounded by the constant $C_2=\mathrm{vol}(T/G)\mathrm{Lip}(\Tilde{h})$. Combining the $G$-equivariance of $\Tilde{f},\Tilde{h}$ with $\Tilde{f}(\Tilde{v})=\Tilde{h}(\Tilde{v})$ we get
\[
[\Tilde{h}(g_1\Tilde{v}),\Tilde{h}(g_2\Tilde{v})]=[g_1\Tilde{h}(\Tilde{v}),g_2\Tilde{h}(\Tilde{v})]=[g_1\Tilde{f}(\Tilde{v}),g_2\Tilde{f}(\Tilde{v})]=[\Tilde{f}(g_1\Tilde{v}),\Tilde{f}(g_2\Tilde{v})].
\]
Therefore, the reduced images $[\Tilde{f}(p)]$, $[\Tilde{h}(p)]$  coincide, up to uniformly bounded error. More precisely, if $C=\max\{C_1,C_2\}$, then for every path $p$ there are path decompositions $[\Tilde{f}(p)]=q_1ms_1$, $[\Tilde{h}(p)]=q_2ms_2$ such that the lengths of $q_1,q_2,s_1,s_2$ are bounded
above by $C$ (which by construction depends only on $\mathrm{Lip}(\Tilde{f})$, $\mathrm{Lip}(\Tilde{h})$ and $\mathrm{vol}(T/G)$).\par
If there are no non-free vertices, we essentially work with homotopy equivalences of finite
graphs and the result is well known (see, for example, \cite{BestvinaHandeltraintracks}).
\end{proof}

\begin{definition}
    Let $T\in\mathcal{O}$ and $\Tilde{f}:T\to T$ be a simplicial $\mathcal{O}$-map. A  (non-trivial) reduced
path $p$ in $T$ is called:
\begin{itemize}
\item \textbf{Nielsen periodic} if $[\Tilde{f}(p)]=gp$ for some $g\in G$.
    \item \textbf{Nielsen preperiodic} if $[\Tilde{f}^n(p)]=gp$ for some $g\in G$ and $n\in\mathbb{N}$.
\end{itemize}
A Nielsen periodic path $p$ is called \textbf{indivisible} if it cannot be written as a concatenation of two Nielsen periodic paths. 
\end{definition}

\subsection{Laminations}
 Attracting and repelling laminations for irreducible automorphisms with exponential growth have been defined in \cite{Lam} in the case of free groups. We follow the approach of \cite{LymanCTs}.\par
 
 Let $\Gamma$ be a marked \textbf{graph of groups} and $T\in \mathcal{O}(\mathcal{G})$ its associated Bass–Serre tree.  More
precisely, $\Gamma$  is a finite graph of groups with trivial edge groups,  its edges are
assigned a positive length, there is an isomorphism $G\to \pi_1(\Gamma)$ well
defined up to composition with inner automorphisms, for each $i\in \{1,\ldots,k\}$ there
is a vertex $v_i$ whose group is conjugate to $G_i$, all other vertex groups are trivial and every terminal vertex is a $v_i$. We write $\partial_{\infty}T$ for the Gromov Boundary of $T$, $V_{\infty}(T)$ for  the collection of vertices of $T$ with infinite valence and $\partial T=\partial_{\infty}T \cup V_{\infty}(T)$. We also write $\partial^2 T$ for $\partial T\times \partial T-\{(x,x),\ x\in\partial T\}$. It is well known that these are independent, up to homeomorphism, of the choice of $T\in \mathcal{O}(\mathcal{G})$. Hence $\partial(G,\mathcal{G})$, $\partial_{\infty}(G,\mathcal{G})$, $\partial^2(G,\mathcal{G})$, $V_{\infty}(G,\mathcal{G})$ can be defined as $\partial T$, $\partial_{\infty} T$, $\partial^2 T$, $V_{\infty}(T)$, respectively, for a $T\in \mathcal{O}(\mathcal{G})$.  Elements of $\partial^2(G,\mathcal{G})$  are called \textbf{algebraic lines}. An algebraic line may be finite, singly
infinite, or bi-infinite, according to whether its endpoints lie in $V_{\infty}(G,\mathcal{G})$ or $\partial_{\infty}(G,\mathcal{G})$. Between any two boundary points $\alpha,\omega\in \partial T$ there is a unique tight path in $T$ called the \textbf{line} $(\alpha,\omega)_T$ from $\alpha$ to $\omega$. Given $T\in \mathcal{O}(\mathcal{G})$, a line in $\partial^2T$ represents an algebraic line $(\alpha,\omega)$ in $\partial^2(G,\mathcal{G})$ if  its endpoints
correspond to  $(\alpha,\omega)$. There is a natural $\mathbb{Z}_2$-action on $\partial^2(G,\mathcal{G})$ given by flipping "coordinates" $(\alpha,\beta)\mapsto (\beta,\alpha)$. From now on, when we write $\partial^2(G,\mathcal{G})$ we shall mean the quotient $\partial^2(G,\mathcal{G})/\mathbb{Z}_2$.
 \begin{definition}
     An \textbf{algebraic lamination} is a closed, $G$-invariant, subset of $\partial^2(G,\mathcal{G})$.  The lines it comprises are the \textbf{lines} or the \textbf{leaves} of the lamination.
 \end{definition}
 We equip $\partial_{\infty}T\cup T$ with a topology, known as the \textbf{observer’s topology}, as follows. Given a point $\Tilde{x}\in T$  a half-tree based at $\Tilde{x}$ is a component of $T\setminus \{\Tilde{x}\}$ together with those boundary points $\xi\in\partial_{\infty}T$ such that the intersection of the tight path $\xi$ with the component contains infinitely many edges. The observer’s topology on $\partial_{\infty}T\cup T$ is generated by the set of half-trees, which form a sub-basis for the topology. Note that the subspace $\partial_{\infty}T\subseteq \partial_{\infty}T\cup T$ inherits the usual topology of $\partial_{\infty}T\cup T$ coming from hyperbolicity of $T$, i.e. a sequence $\Tilde{x}_n\in \partial_{\infty}T\cup T,\ n\in\mathbb{N}$ converges to $\xi\in \partial_{\infty}T$ if
and only if for every finite subpath $\Tilde{\gamma}$ of the ray determining $\xi$, the tight path or ray, starting from a fixed basepoint and, ending at $\Tilde{x}_n$ contains $\Tilde{\gamma}$ for all large $n$. Given a finite tight edge path $\widetilde{\gamma}\in T$, the 
neighborhood $N(\widetilde{\gamma})\subseteq  \partial^2 T$ is defined to be the set of lines that contain $\widetilde{\gamma}$ as a subpath. The sets $N(\widetilde{\gamma})$ form  a basis for the bi-infinite lines. \par
 We denote by $\partial^2 \Gamma$ the image of $\partial^2 T$ under the natural projection $T\to T/G=\Gamma$. We give $\partial^2 \Gamma$ the quotient topology. For a tight edge path $\gamma$ in $\Gamma$, $N(\gamma)$ is defined to be the set of those lines in $\Gamma$ that contain $\gamma$ as a
subpath. The sets $N(\gamma)$ form a basis for any bi-infinite line in $\partial^2\Gamma$.\par
If $\Tilde{f}:T\to T'$ is an $\mathcal{O}$-map, then there is a natural homeomorphism $\hat{f}:\partial T\to \partial T'$ which induces a homeomorphism $\Tilde{f}_{\#}:\partial^2 T\to \partial^2 T'$ and a homeomorphism $f_{\#}:\partial^2\Gamma\to\partial^2\Gamma'$, where $\Gamma$ and $\Gamma'$ are the quotient graphs corresponding to $T$ and $T'$ respectively.\par
 
An element of $G$ is called \textbf{peripheral} if it stabilises a vertex, i.e. if it is contained in a conjugate of a vertex group of some
(and hence any) tree $T\in \mathcal{O}(\mathcal{G})$. The conjugacy class of a nonperipheral element of $G$ determines a periodic bi-infinite line in $\partial^2(G,\mathcal{G})$. More precisely, $\partial_{\infty}(G,\mathcal{G})$ can be identified
with the set of (simply) infinite reduced words in the free product decomposition corresponding to $\mathcal{G}$. For any nonperipheral element $g\in G$ we define the infinite
words $g^{+\infty}=\lim\limits_{n\to+\infty}g^n$ and $g^{-\infty}=\lim\limits_{n\to+\infty}g^{-n}$ and thus we derive the line $(g^{-\infty},g^{+\infty})\in \partial^2(G,\mathcal{G})$. We say that a line in $\partial^2(G,\mathcal{G})$  is \textbf{carried} by the conjugacy class of a free factor $A$ of $G$ of the form $A=G_{i_1}\ast\ldots\ast G_{i_{\nu}}\ast F_{\lambda}$, $i_j\in \{1,\ldots,k\}$, if it is
in the closure of the periodic lines in $\partial^2(G,\mathcal{G})$ determined by the conjugacy classes of nonperipheral
elements of $A$.\par
 Let $\ell_n$ be a sequence of bi-infinite lines in $\partial_{\infty}(G,\mathcal{G})$. We note that $\ell_n$ converges to $\ell\in\partial_{\infty}(G,\mathcal{G})$ in the observer's topology if and only if for every $M\in\mathbb{N}$ and every finite segment $P\subseteq \ell$ of length $M$, there are elements $g_n\in G$ and segments $Q_n\subseteq \ell_n$ such that $g_nQ_n\supseteq P$ for all large $n\in\mathbb{N}$.  \par
Let $[\phi]\in \mathrm{Out}(G,\mathcal{G})$ and $\ell,\ell'$ two bi-infinite lines in $\partial^2(G,\mathcal{G})$. We say that $\ell'$ is \textbf{weakly attracted} to $\ell$ under the action of $[\phi]$ if there is a topological representative $\Tilde{f}:T\to T$ of $[\phi]$ such that the sequence $\Tilde{f}_{\#}^k(\ell')$ converges to $\ell$ as $k$ tends to infinity. We will also write $[\phi]_{\#}^k(\ell')\to\ell$ instead of $\Tilde{f}_{\#}^k(\ell')\to \ell$. 
\begin{remark}
    Note that the convergence $[\phi]_{\#}^k(\ell')\to\ell$ is independent of the chosen topological representative $\Tilde{f}$ of $[\phi]$. Indeed, let $i_g:T\to T$, $t\mapsto gt$ be the isometry corresponding to the inner automorphism $\tau_g(h)=ghg^{-1}$ and $i_g\circ \Tilde{f}:T\to T$ be another topological representative of $[\phi]$. If the sequence $x_k=\Tilde{f}_{\#}^k(\ell')$ converges to $\ell$, we will prove that the sequence $y_k=(i_g\circ \Tilde{f})_{\#}^k(\ell')$ converges to $\ell$, too. Let $M\in\mathbb{N}$ and a segment $P\subseteq \ell$ of length $M$. There are elements $g_n\in G$ and segments $Q_n\subseteq \Tilde{f}_{\#}^n(\ell')$ such that $g_nQ_n\supseteq P$ for all large $n\in\mathbb{N}$. We also have, $(i_g\circ \Tilde{f})^n=i_{\underbrace{\tiny{g\cdot\phi(g)\cdot\ldots\cdot\phi(\ldots \phi(g)\ldots))}}_{n}}\circ \Tilde{f}^n$, hence $Q_n\subseteq \Tilde{f}_{\#}^n(\ell')$ implies that $g\cdot\phi(g)\cdot\ldots\cdot\phi(\ldots \phi(g)\ldots))Q_n=i_{\underbrace{g\cdot\phi(g)\cdot\ldots\cdot\phi(\ldots \phi(g)\ldots))}_{n}}(Q_n)\subseteq (i_g\circ\Tilde{f})_{\#}^n(\ell')$ which we may write as $ h_nQ_n\subseteq (i_g\circ\Tilde{f})_{\#}^n(\ell')$ for $h_n=\underbrace{\tiny{g\cdot\phi(g)\cdot\ldots\cdot\phi(\ldots \phi(g)\ldots))}}_{n}\in G$. To conclude that $y_k\to \ell$ we set $T_n\coloneqq h_nQ_n$ and $b_n\coloneqq g_nh_n^{-1}\in G$. We now have that $T_n\subseteq (i_g\circ \Tilde{f})^n_{\#}(\ell)$ and for all large $n\in\mathbb{N}$
\begin{align*}
        g_nQ_n\supseteq P\Rightarrow \underbrace{g_nh_n^{-1}}_{b_n}\cdot \underbrace{h_nQ_n}_{T_n}\supseteq P \Rightarrow b_n\cdot T_n\supseteq P.
    \end{align*}
\end{remark}
A subset $U\subseteq \partial^2(G,\mathcal{G})$ is called an \textbf{attracting neighborhood}
of $\ell\in \partial^2(G,\mathcal{G})$ for the action of $[\phi]$ if $\Tilde{f}_{\#}(U)\subseteq U$ and if $\{\Tilde{f}^k_{\#}(U),\, k\ge0\}$ is a neighborhood basis for $\ell\in \partial^2(G,\mathcal{G})$. A line $\ell$ is called \textbf{birecurrent} if it is bi-infinite and every subpath of $\ell$ occurs
infinitely often as a subpath of each end of $\ell$.
\begin{definition}
    Let $[\phi]\in\mathrm{Out}(G,\mathcal{G})$ be an outer automorphism of $G$. A closed, $G$-invariant subset $\Lambda^+\subseteq \partial^2(G,\mathcal{G})$ is called an \textbf{attracting lamination} for $[\phi]$ if it is the closure of a single
line $\ell$ such that:
\begin{enumerate}
    \item The line $\ell$ is birecurrent.
    \item The line $\ell$ has an attracting neighborhood for the action of some iterate of $[\phi]$.
    \item The line $\ell$ is not carried by any $[\phi]$-periodic $\mathbb{Z}$ or $\mathbb{Z}_2\ast \mathbb{Z}_2$ free factor.
\end{enumerate}
The line $\ell$ is said to be \textbf{generic} for $\Lambda^+$. The set of attracting laminations for $[\phi]\in\mathrm{Out}(G,\mathcal{G})$ is denoted by $\mathcal{L}([\phi])$.

\begin{lemma}[{{\cite[Lemma 4.2, Lemma 4.6]{LymanCTs}}}]
    Let $[\phi]\in\mathrm{Out}(G,\mathcal{G})$ be an outer automorphism of $G$. The set of attracting laminations $\mathcal{L}([\phi])$ is $[\phi]$-invariant and finite. 
\end{lemma}

\end{definition}
We now focus on irreducible with irreducible powers (iwip) outer automorphisms. We shall give a more constructive definition for attracting laminations which will allow us to prove some extra properties that are needed for our purposes. Let $[\phi]\in\mathrm{Out}(G)$ be an iwip automorphism and $\Tilde{f}:T\to T$ a train track representative of $[\phi]$. Then, by Perron-Frobenius theory, there is a uniform constant $\lambda>1$ such that for every edge $\widetilde{e}\in E(T)$, $\Tilde{f}$ expands the length of $\widetilde{e}$ by $\lambda$. Changing $\Tilde{f}$ with some iterate, we may assume that there is a point $x\in T$ in the interior of some edge $\widetilde{e}$ such that $\Tilde{f}(x)=gx$ for some $g\in G$. Therefore, composing with the isometry corresponding to the inner automorphism $x\mapsto g^{-1}xg$, we may further assume that there is a point $x\in T$ in the interior of some edge $\widetilde{e}$ such that $\Tilde{f}(x)=x$. For some sufficiently small $\epsilon>0$, the $\epsilon$-neighbourhood $U$ of $x$ remains in the interior of the edge $\widetilde{e}$ and $\Tilde{f}^n(U)\supseteq U$ for every $n\in\mathbb{N}$. Any isometry $\ell:(-\epsilon,\epsilon)\to U$ with $\ell(0)=x$ can be extended to a unique isometry $\ell:\mathbb{R}\to T$ such that $\ell(\lambda^nt)=\Tilde{f}^n(\ell(t))$ and we say that the bi-infinite line $\ell$ is
obtained by iterating a neighbourhood of $x$.\par
The following pair of definitions are going to be used in the definition of the stable lamination associated to an iwip outer automorphism $[\phi]\in\mathrm{Out}(G,\mathcal{G})$. 
\begin{definition}
    \begin{enumerate}
        \item We say that two isometric immersions $C:[a,b]\to T$ and $D:[c,d]\to T$, where $a,b,c,d\in\mathbb{R}$ are equivalent, if there exists an isometry $q:[a,b]\to [c,d]$ such that $D\circ q=C$. Note
that this relation is an equivalence relation on the set of isometric immersions
from a finite interval to $T$.
        \item Let $P$ be an equivalence class of an isometric immersion, $\gamma:[a,b]\to T$ be a
representative of $P$ and $\Tilde{f}:T\to T$ a train track representative as above. We define $\Tilde{f}(P)$ to be the equivalence class of $\Tilde{f}\circ \gamma:[a,b]\to T$, (where we consider the reduced image $[\Tilde{f}(\gamma)]$ after re-scaling in order to become an isometric immersion).
\item A \textbf{line segment} (or \textbf{leaf segment}) of an isometric immersion $\ell:\mathbb{R}\to T$ is the equivalence class of the
restriction to a finite interval. 
    \end{enumerate}
\end{definition}
To any isometric immersion $\ell:\mathbb{R}\to T$ we associate the set $I_{\ell}$ of the line
segments of $\ell$.
\begin{definition}
    Let $\ell,\ell':\mathbb{R}\to T$ be two isometric immersions. We say that $\ell,\ell'$ are
\textbf{equivalent}, if for every line segment $P$ of $\ell$ there is an element $g\in G$ and a line segment $Q$ of $\ell'$ such that $P=gQ$ and vice versa.
\end{definition}
It is well known that if we construct any other bi-infinite line by iterating
a neighbourhood of any other periodic point, then it must be equivalent to $\ell$ as constructed above.

\begin{definition}
    Let $\Tilde{f}:T\to T$ be a train track representative of $[\phi]$. The \textbf{stable lamination} in $T$-coordinates associated to $\Tilde{f}$, denoted by $\Lambda=\Lambda_{\Tilde{f}}^+(T)$, is the equivalence class of isometric immersions
from $\mathbb{R}$ to $T$ containing some (and hence any) immersion obtained, as above
(by iterating a neighborhood of a periodic point). The
immersions representing $\Lambda$ are called \textbf{lines} or \textbf{leaves} and the line segments of some line of $\Lambda$ are
called \textbf{line} (or leaf) \textbf{segments} of the stable lamination $\Lambda$.
\end{definition}

\begin{definition}
    Two paths $p,q$ in a $G$-tree $T$ are called \textbf{projectively equivalent} if they project to the same path in the quotient graph $T/G$.
\end{definition}

\begin{definition}
An isometric immersion $\ell:\mathbb{R}\to T$ is called \textbf{quasiperiodic}, if for every $L>0$ there exists $L'>0$  such that for every line segment $P$ of $\ell$ of length $L$ and for every line segment $Q$ of length $L'$, there is some $g\in G$ and a line segment $P'$ which is projectively
equivalent to $P$ such that $gP'\subseteq Q$.
\end{definition}

The following remarks are well known to the experts.

\begin{remark}
    Every line of a stable lamination $\Lambda_{\Tilde{f}}^+(T)$ is quasiperiodic.
\end{remark}

\begin{remark}\label{rem4.4}
    There
is no line $\ell$ of the lamination $\Lambda_{\Tilde{f}}^+(T)$ which is "periodic" in the sense that $\ell$ can not be of the form $\ldots P_{-1}P_0P_1\ldots$ where every segment $P_i$ is projectively equivalent to some fixed segment $P$.
\end{remark}

Let $\Tilde{f}:T\to T$ be
a train track representative of $[\phi]\in\mathrm{Out}(G,\mathcal{G})$ and let $\Lambda_{\Tilde{f}}^+(T)$  be the associated stable lamination. If $T'$ is some other point of $\mathcal{O}$, then there is an $\mathcal{O}$-map $\tau:T\to T'$ (see \autocite[Lemma~4.2]{francaviglia2015stretching}).  For any isometric immersion $\ell:\mathbb{R}\to T$ we denote
by $\tau(\ell):\mathbb{R}\to T'$ the isometric immersion corresponding to the reduced image of the composition $\tau\circ \ell$. This isometric immersion is unique up to precomposition
with an isometry of $\mathbb{R}$.\par
The following lemma seems also to be well known to experts, but we include a sketch of proof for completeness reasons.
\begin{lemma}[{{\cite[Lemma 3.13]{syrigos2016irreducible}}}]
Let $[\phi]\in\mathrm{Out}(G,\mathcal{G})$, $T,T'\in\mathcal{O}$, $\tau:T\to T'$ an $\mathcal{O}$-map, $\Tilde{f}:T\to T$ a train track representative of $[\phi]$ and $\Lambda_{\Tilde{f}}(T)$ the stable lamination associated to $\Tilde{f}$ in $T$-coordinates.
    \begin{enumerate}
        \item If $\ell,\ell':\mathbb{R}\to T$ are equivalent lines of $\Lambda_{\Tilde{f}}(T)$, then $\tau(\ell)$, $\tau(\ell')$ are equivalent lines in $T'$.
        \item If $\ell$ is quasiperiodic, then $\tau(\ell)$ is also  quasiperiodic.
    \end{enumerate}
\end{lemma}
\begin{proof}[Sketch of proof]
    We factor $\tau:T\to T'$ as a finite composition $\tau=\sigma\circ f$ of an $\mathcal{O}$-permutation $\sigma$ and a sequence of $\mathcal{O}$-foldings $f$ (see \autocite[Proof of Proposition~7.5]{francaviglia2015stretching}). It is enough to prove that the lemma
is true for $\mathcal{O}$-permutations and $\mathcal{O}$-foldings. For an $\mathcal{O}$-permutation $\sigma$ parts $1$ and $2$ follow immediately, from the definitions and the observation that $[\sigma(\ell)]=\sigma(\ell)$. \par
We suppose now that $f$ is an equivariant isometric fold of some segments $p,q$ starting from the same point $i(p)=i(q)=\Tilde{v}$. We denote by $r$ the image of $p,q$ after the fold. If $P'$ is a line segment of $f(\ell)$  which contains some elements in the orbit of $r$, then there exists some line segment $P$ which contains the same number of elements in the orbit of $r$ such that $[f(P)]=P'$. Since $\ell,\ell'$ are equivalent, there is some $g\in G$ such that $gP'\subseteq \ell'$ hence $[f(gP)]$ is a line segment of $[f(\ell')]$. Since $[f(gP)]$ is a  translation of $P'$,  it follows that $[f(\ell)],[f(\ell')]$ are equivalent.\par
To prove the quasiperiodity of $[f(\ell)]$, we fix $L>0$ as in the definition. Denote by $M$ the
maximum number of elements in the orbit of $\Tilde{v}$ which are contained in a line segment of length $L$. Denote by $L'$ the number corresponding, by quasiperiodicity, to $L''=L+2M\cdot\mathrm{length}(p)$. If $P'$ is a line segment  of length $L$, then there is some leaf segment $P$ containing the same number of elements in the orbit of the folded turn such that $[f(P)]=P'$. It follows by construction, that $P$ has length at most $L''$ and so some
projectively equivalent segment $P''$ is contained in every line segment of $\ell$ of length $L'$. For any line segment $Q$ of $[f(\ell)]$ of length $L'$, its pre-image has length at least $L'$  and
therefore the pre-image has the required property. Therefore $Q$ contains a projectively
equivalent path to $P'$ and the quasiperiodicity follows. 
\end{proof}

\begin{definition}
    Let $[\phi]\in\mathrm{Out}(G,\mathcal{G})$, $T,T'\in\mathcal{O}$, $\tau:T\to T'$ an $\mathcal{O}$-map, $\Tilde{f}:T\to T$ a train track representative of $[\phi]$ and $\Lambda_{\Tilde{f}}(T)$ the stable lamination associated to $\Tilde{f}$ in $T$-coordinates. The stable lamination of $\Tilde{f}$ in $T'$-coordinates, which we denote by $\Lambda_{\Tilde{f}}(T')$, is the class of equivalent lines containing $\tau(\ell)$ for some (and by previous lemma any) line of $\Lambda_{\Tilde{f}}(T)$.
\end{definition}

It can be shown \cite[Lemma 3.16]{syrigos2016irreducible} that if $\Tilde{f}:T\to T$, $\Tilde{h}:T'\to T'$ are two train track representatives of $[\phi]\in\mathrm{Out}(G,\mathcal{G})$, then the stable lamination of $\Tilde{f}$ in $T'$-coordinates coincides with the stable lamination of $\Tilde{h}$ in $T'$-coordinates i.e. $\Lambda_{\Tilde{f}}(T')=\Lambda_{\Tilde{h}}(T')$. The definition of the stable
lamination associated to $[\phi]$, that follows,  does not depend on the chosen
representative, but only on $[\phi]$.

\begin{definition}
   Let $[\phi]\in\mathrm{Out}(G,\mathcal{G})$ be an outer automorphism of $G$ and $\Tilde{f}$ a train track representative of $[\phi]$. The \textbf{stable lamination} $\Lambda_{[\phi]}^+$ associated to $[\phi]$  is the collection
   \[
   \Lambda_{[\phi]}^+\coloneqq \{\Lambda_{\Tilde{f}}^+(T),\ T\in\mathcal{O}\}.
   \]
   The \textbf{unstable lamination} $\Lambda_{[\phi]}^-$ associated to $[\phi]$ is the stable lamination of $[\phi]^{-1}$. 
\end{definition}

We denote by $\mathcal{IL}$ the set of stable
laminations $\Lambda_{[\phi]}$ as $[\phi]$ ranges over all iwip automophisms of $\mathrm{Out}(G,\mathcal{G})$. The group $\mathrm{Out}(G,\mathcal{G})$ acts on the set $\mathcal{IL}$ via
\[
[\psi]\cdot \Lambda_{[\phi]}^+=\Lambda^+_{[\psi]\cdot [\phi]\cdot [\psi]^{-1}}.
\]
More specifically, if $\ell\subseteq T$ is a line of $\Lambda_{[\phi]}^+$ and $\Tilde{h}:T\to T$ is an $\mathcal{O}$-representative of $[\psi]$, then the reduced image $[\Tilde{h}(\ell)]$ represents a leaf of $\Lambda^+_{[\psi]\cdot [\phi]\cdot [\psi]^{-1}}$.  Since for any iwip outer automorphism $[\phi]\in\mathrm{Out}(G,\mathcal{G})$ and for any $[\psi]\in\mathrm{Out}(G,\mathcal{G})$, the outer automorphism $[\psi]\cdot [\phi]\cdot [\psi]^{-1}$ 
is iwip relative to $\mathcal{O}$, the previous action
is well defined. \par
The \textbf{stabiliser} of the stable lamination $\Lambda_{[\phi]}^+$ is the subgroup
\[
\mathrm{Stab}(\Lambda_{[\phi]}^+)=\{[\psi]\in  \mathrm{Out}(G,\mathcal{G}),\ [\psi]\cdot \Lambda_{[\phi]}^+=\Lambda_{[\phi]}^+ \}\leq \mathrm{Out}(G,\mathcal{G}).
\]
Note that the \textbf{relative
centraliser} $C([\phi])=\{[\psi]\in\mathrm{Out}(G,\mathcal{G}),\, [\psi]\cdot [\phi]=[\phi]\cdot [\psi]\}$ of $[\phi]$ in $\mathrm{Out}(G,\mathcal{G})$  is a subgroup of
$\mathrm{Stab}(\Lambda_{[\phi]}^+)$.
\section{The first main theorem}\label{mainthmsection3}
The following lemma is not essential for our purposes but it is of independent interest.
\begin{lemma}
      If $[\phi]\in \mathrm{Out}(G,\mathcal{G})$ is an outer automorphism such that $\phi(G_3\ast\cdots\ast G_k\ast F_p)\subseteq G_3\ast\cdots\ast G_k\ast F_p$ and $[\phi]([G_i])\subseteq [G_i]$ for every $i=1,\ldots,k$, then, for every $g_1\in G_1$ and $g_2\in G_2$, the image $\phi(g_1g_2)$ contains at most one element of $G_1$ or at most one element of $G_2$.
\end{lemma}
\begin{proof}
       Since $[\phi]([G_i])\subseteq [G_i]$ for every $i=1,\ldots,k$, there are $t,x\in G$ such that $\phi(G_1)\subseteq xG_1x^{-1}$ and $\phi(G_2)\subseteq tG_2t^{-1}$.
         We may also assume \autocite[Lemma~3.2]{SykFixedSubgroups} that $\phi(G_1)=xG_1x^{-1}$ and $\phi(G_2)=tG_2t^{-1}$. The restriction of $\phi$ on $G_3\ast\cdots\ast G_k\ast F_p$ can be extended to an automorphism $\phi'$ of $G$ by
        \[
        \begin{matrix}
g_1\mapsto g_1,& \text{for all }g_1\in G_1,\\
g_2\mapsto g_2,& \text{for all }g_2\in G_2.
\end{matrix}
        \]
        After replacing $\phi$ with $\phi\circ (\phi')^{-1}$ we may assume that the restriction of $\phi$ on $G_3\ast\cdots\ast G_k\ast F_p$ is the identity. To sum up, the action of $\phi$ on $G$ is given by
        \[
         \begin{cases}
g_1\mapsto x\phi_1(g_1)x^{-1},& \text{for all }g_1\in G_1,\text{ where } \phi_1\in \mathrm{Aut}(G_1)\\
g_2\mapsto t\phi_2(g_2)t^{-1},& \text{for all }g_2\in G_2,\text{ where } \phi_2\in \mathrm{Aut}(G_2)\\
g_k\mapsto g_k,& \text{for all } g_k\in G_k,\, k\not\in\{1,2\}\\
x_i\mapsto x_i,& \text{for all } i\in\{1,\ldots,p\}.
\end{cases}
        \]

        Let $\Gamma$ be the graph of groups described in Figure \ref{fig:002} below, where $G_{v_{i}}=G_i$ and in the central vertex $w$ we associate the trivial group.

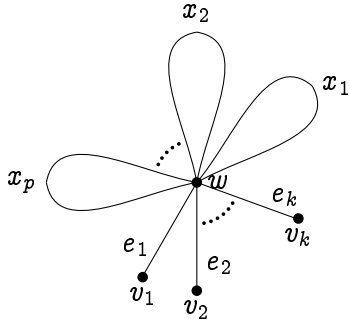
\begin{figure}[H]
\begin{center}
\begin{tikzpicture}
\draw (0,0) to [out=100,in=-170] (0,2);
\draw (0,0) to [out=80,in=-10] (0,2);
\fill (0,0) circle (2pt);
\fill (-0.72,-1.26) circle (2pt);
\fill (0,-1.44) circle (2pt);
\fill (1.35,-0.48) circle (2pt);
\node[below] at (-0.72,-1.26) {$v_1$};
\node[below] at (0,-1.44) {$v_2$};
\node[below] at (1.35,-0.48) {$v_k$};
\node[right] at (0,0) {$w$};
\begin{scope}[rotate=-50]
\draw (0,0) to [out=100,in=-170] (0,2);
\draw (0,0) to [out=80,in=-10] (0,2);
\end{scope}
\begin{scope}[rotate=90]
\draw (0,0) to [out=100,in=-170] (0,2);
\draw (0,0) to [out=80,in=-10] (0,2);
\end{scope}
\node[above] at (0,2) {$x_2$};
\node[right] at (1.532,1.285) {$x_1$};
\node[left] at (-2,0) {$x_p$};
\node[right] at (0,-1.1) {$e_2$};
\node[left] at (-0.5,-0.866) {$e_1$};
\node[left] at (1.5,-0.2) {$e_k$};
\begin{scope}[rotate=0]
\fill (-0.4,0.4) circle (0.7pt);
\end{scope}
\begin{scope}[rotate=10]
\fill (-0.4,0.4) circle (0.7pt);
\end{scope}
\begin{scope}[rotate=20]
\fill (-0.4,0.4) circle (0.7pt);
\end{scope}
\begin{scope}[rotate=-10]
\fill (-0.4,0.4) circle (0.7pt);
\end{scope}
\begin{scope}[rotate=-20]
\fill (-0.4,0.4) circle (0.7pt);
\end{scope}
\begin{scope}[rotate=60]
\fill (0,-0.55) circle (0.7pt);
\fill[rotate=-10] (0,-0.55) circle (0.7pt);
\fill[rotate=-20] (0,-0.55) circle (0.7pt);
\fill[rotate=-30] (0,-0.55) circle (0.7pt);
\fill[rotate=-40] (0,-0.55) circle (0.7pt);
\fill[rotate=-50] (0,-0.55) circle (0.7pt);
\end{scope}
\draw (0,0) -- (0,-1.4);
\draw[rotate=-30] (0,0) -- (0,-1.4);
\draw[rotate=70] (0,0) -- (0,-1.4);
\end{tikzpicture}
\end{center}
 \caption{The graph of groups $\Gamma$}
 \label{fig:002}
\end{figure}

        Let $T$ be the Bass-Serre tree corresponding to $\Gamma$ and $\Tilde{f}:T\to T$ the topological representative of $\phi$  which is given by (we extend naturally in the rest of the vertices and the edges of $T$)
        \[
        \Tilde{f}:\begin{cases}
        \widetilde{w}\mapsto\widetilde{w}\\
            \widetilde{v}_i\mapsto\widetilde{v}_i,& \text{ for all } i\in\{3,\ldots,k\}\\
            \widetilde{v}_1\mapsto x\widetilde{v}_1\\  
            \widetilde{v}_2\mapsto t\widetilde{v}_2
        \end{cases}
        \]
      Suppose, toward a contradiction, that $\phi(g_1g_2)$ contains more than two elements of both $G_1,G_2$. This implies that the reduced image $\Tilde{f}(\widetilde{e}_2\cdot\widetilde{e}_1^{-1})$ crosses at least two edges from the set $G\cdot\widetilde{e}_1^{\pm}\coloneqq G\cdot\widetilde{e}_1\cup G\cdot \widetilde{e}_1^{-1}$ and at least two edges from the set $G\cdot\widetilde{e}_2^{\pm}\coloneqq G\cdot\widetilde{e}_2\cup G\cdot \widetilde{e}_2^{-1}$. We note that whenever we fold two edges we will still use the same notation for the resulting paths after the fold. There are two cases:\par
\textbf{Case $1$}: $\Tilde{f}(\widetilde{e}_1)$ contains only one element of the orbits $G\cdot\widetilde{e}_1^{\pm}$ or $\Tilde{f}(\widetilde{e}_2)$ contains only one element of the orbits $G\cdot\widetilde{e}_2^{\pm}$. (1)\par
Suppose that $\Tilde{f}(\widetilde{e}_1)$ contains only one element of the orbits $G\cdot\widetilde{e}_1^{\pm}$. Note that  $\Tilde{f}(\widetilde{e}_2)$ may contain several elements of both $G\cdot\widetilde{e}_1^{\pm}$ and $G\cdot\widetilde{e}_2^{\pm}$. After suitable folds we may assume that $\Tilde{f}(\widetilde{e}_1)=\widetilde{e}_1$ and, after probably some more folds, that $\Tilde{f}(\widetilde{e}_2)$ begins either with the direction of $\widetilde{e}_2$ or the path $\widetilde{e}_1$ (here we still denote by $\widetilde{e}_1$ and $\widetilde{e}_2$ the resulting paths, see also the discussion above). If $\Tilde{f}(\widetilde{e}_2)$ begins with the first edge of $\widetilde{e}_2$, then since there are no available folds, $\Tilde{f}$ should be an isometry ($\Tilde{f}$ can be written as a finite composition of $\mathcal{O}$-foldings followed by an $\mathcal{O}$-permutation, see \autocite[Proof of Proposition~7.5]{francaviglia2015stretching}) which is a contradiction ($\Tilde{f}(\widetilde{e}_2)$ contains more edges than $\widetilde{e}_2$ itself).   If $\Tilde{f}(\widetilde{e}_2)$ starts with $\widetilde{e}_1$, then we fully fold $\widetilde{e}_1$ over $\widetilde{e}_2$. There are no other folds, since the folded path $\widetilde{e}_1$ goes to itself, hence $\Tilde{f}$ is an isometry which is a contradiction.

 \textbf{Case $2$}: Both $\Tilde{f}(\widetilde{e}_1)$  and $\Tilde{f}(\widetilde{e}_2)$ contain at least two elements of the set $G\cdot\widetilde{e}_1^{\pm}\cup G\cdot\widetilde{e}_2^{\pm}$.

        If the $\Tilde{f}$ image of $\widetilde{e}_1$ (resp. $\widetilde{e}_2$) begins  with an edge in the orbit of $T_1\coloneqq\{\widetilde{e}_3,\ldots,\widetilde{e}_k,\widetilde{x}_1,\ldots,\widetilde{x}_p\}$, then we may fold an initial segment of $\widetilde{e}_1$ (resp.$ \widetilde{e}_2$) into $T_1$. Note also that $\Tilde{f}(\widetilde{e}_1)=[\widetilde{w},x\widetilde{v}_1]$ (resp. $\Tilde{f}(\widetilde{e}_2)=[\widetilde{w},t\widetilde{v}_2]$) can not begin nor end with an element in the orbit of $\widetilde{e}_1^{-1}$ or $\widetilde{e}_2^{-1}$, since the final vertices of the corresponding paths is $\widetilde{v}_1$, (here we use the assumption that $[\phi]([G_i])=[G_i],\ i=1,2$). Therefore, we may assume that $\Tilde{f}(\widetilde{e}_1)$ (resp. $\Tilde{f}(\widetilde{e}_2)$) begins with an element in 
$G\cdot\widetilde{e}_1^+\cup G\cdot\widetilde{e}_2^+$, i.e. the orbit of $\widetilde{e}_1$ or $\widetilde{e}_2$. We shall now prove that $\Tilde{f}$ is an $\mathcal{O}$-permutation and thus we will have the desired contradiction (to our assumption that both $\Tilde{f}(\widetilde{e}_i),\ i=1,2$ have simplicial length at least $2$). Following \autocite[Proof of Proposition~7.5]{francaviglia2015stretching} we may write $\Tilde{f}$
 as a composition $\Tilde{f}=i\circ p$ where $p$ is a finite composition of $\mathcal{O}$-foldings and $i$ an $\mathcal{O}$-permutation. As $\widetilde{e}_1,\widetilde{e}_2$ begin with the same vertex (as paths) we can still talk about folding at the direction of these paths  \autocite[Proposition~7.5]{francaviglia2015stretching}. We distinguish three cases regarding the first edges of $\Tilde{f}(\widetilde{e}_1)$ and $\Tilde{f}(\widetilde{e}_2)$. 
 
\begin{itemize}
    \item If the path $\Tilde{f}(\widetilde{e}_1)$ begins with an element in the orbit of the first edge of $\widetilde{e}_1$ and $\Tilde{f}(\widetilde{e}_2)$ begins with an element in the orbit of the first edge of $\widetilde{e}_2$, then by construction there are no $\mathcal{O}$-foldings, hence $\Tilde{f}$ is an $\mathcal{O}$-permutation.
    \item If the path $\Tilde{f}(\widetilde{e}_1)$ begins with an element in the orbit of $\widetilde{e}_2$ and $\Tilde{f}(\widetilde{e}_2)$ begins with an element in the orbit of $\widetilde{e}_1$, then, again by construction, there are no $\mathcal{O}$-foldings, hence $\Tilde{f}$ is an $\mathcal{O}$-permutation.  
    \item  Suppose now that both $\Tilde{f}(\widetilde{e}_1)$ and $\Tilde{f}(\widetilde{e}_2)$ begin with an element in the orbit of the first edge of the same $\widetilde{e}_i$, say $\widetilde{e}_1$, i.e. $\Tilde{f}(\widetilde{e}_1)=\widetilde{e}_1\cdot g_{11}\widetilde{e}_1^{-1}\cdot p$ and $\Tilde{f}(\widetilde{e}_2)=\widetilde{e}_1\cdot g_{21}\widetilde{e}_1^{-1}\cdot q$ where it may be $g_{11}\not= g_{12}$. After performing the available fold, either the resulting graph $T/G$ has one vertex more or one of the paths $\widetilde{e}_i,\ i=1,2$ gets fully folded on the other path and we proceed as in the very first case (1). If the resulting graph $T/G$ has one vertex more, then, since there are no other available folds, the resulting map should be an $\mathcal{O}$-permutation which is a contradiction.
\end{itemize}
\end{proof}

\begin{lemma}\label{mainlemma}
    \begin{enumerate}
       \item 
        If $[\phi]\in \mathrm{Out}(G,\mathcal{G})$ is an automorphism such that $\phi(G_2\ast\cdots\ast G_k\ast F_p)\subseteq G_2\ast\cdots\ast G_k\ast F_p$ and $[\phi]([G_i])\subseteq [G_i]$ for every $i=1,\ldots,k$, then, for every $g_1\in G_1$, the image $\phi(g_1)$ contains at most one element of $G_1$.
        \item 
         If $[\phi]\in \mathrm{Out}(G,\mathcal{G})$ is an automorphism such that $\phi(G_1\ast \cdots \ast G_k\ast F(\{x_1,\ldots,x_{j-i},x_{j+1},\ldots,x_p\}))\subseteq G_1\ast \cdots\ast  G_k\ast F(\{x_1,\ldots,x_{j-i},x_{j+1},\ldots,x_p\})$ and $[\phi]([G_i])\subseteq[G_i]$ for every $i=1,\ldots,k$, then, the image $\phi(x_j)$ contains at most (in fact exactly) one occurrence of $x_j$.
       
     \end{enumerate}
\end{lemma}
\begin{proof}
    \begin{enumerate}
       
 \item As in the case above, we may assume that $\phi(G_1)=xG_1x^{-1}$ for some $x\in G$ and that the action of $\phi$ on $G$ is given by
 \[
 \begin{cases}
     g_1\mapsto x\phi_1(g_1)x^{-1},& \text{for all } g_1\in G_1, \text{ where } \phi_1\in\mathrm{Aut}(G_1)\\
     g_k\mapsto g_k,& \text{for all } g_k\in G_k,\ k\not=1\\
     x_i\mapsto x_i,& \text{for all } i\in \{1,\ldots,p\}
 \end{cases}
 \]
 Let $\Gamma$ and $T$ be the graph of groups and corresponding Bass-Serre tree as in the above case. We consider the topological representative $\Tilde{f}:T\to T$ of $\phi$ defined by
 \[
 \Tilde{f}:\begin{cases}
     \widetilde{w}\mapsto \widetilde{w} \\
     \widetilde{v}_i\mapsto \widetilde{v}_i,& \text{for all } i\in \{2,\ldots,k\}\\
     \widetilde{v}_1\mapsto x\widetilde{v}_1
 \end{cases}
 \]
 After folding we may assume that the reduced image $\Tilde{f}(\widetilde{e}_1)=[\widetilde{w},x\widetilde{v}_1]$ begins and ends with elements of $G\cdot\widetilde{e}_1^+$. Now, we may write $\Tilde{f}$
 as a composition $\Tilde{f}=i\circ p$ where $p$ is a finite composition of $\mathcal{O}$-foldings and $i$ an $\mathcal{O}$-permutation. Since there are no available $\mathcal{O}$-foldings (in the resulting graph-map) $\Tilde{f}$ is an $\mathcal{O}$-permutation hence $\Tilde{f}(\widetilde{e}_1)=h\cdot\widetilde{e}_1$, for some $h\in G$, which implies that the image $\phi(g_1)$ may contain at most one element of $G_1$, for all $g_1\in G_1$.
 \item Again, we may assume that the action of $\phi$ on $G$ is given by
 \[
 \begin{cases}
     g_i\mapsto g_i,& \text{for all } g_i\in G_i,\ i=1,\ldots,k\\
     x_i\mapsto x_i,& \text{for all } i\in \{1,\ldots,p\}\setminus \{j\}\\
     x_j\mapsto \phi(x_j)
 \end{cases}
 \]
 Let $\Gamma$, $T$ the graph of groups and tree that were used in the above cases and the topological representative $\Tilde{f}:T\to T$ of $\phi$ which is induced by the identity on the vertices of the tree of representatives
\[
 \begin{cases}
\widetilde{w}\mapsto  \widetilde{w} \\
  \widetilde{v}_i\mapsto  \widetilde{v}_i,& \text{for all } i\in \{1,\ldots,k\} \end{cases}
 \]
 We focus on the image of the edge $\widetilde{x}_j=[\widetilde{w},x_j\cdot \widetilde{w}]$. After folding appropriately, we may assume that the reduced path $\Tilde{f}(\widetilde{x}_j)$ begins and ends with elements of the orbits $G\cdot\widetilde{x}_j^{\pm}$. We may write $\Tilde{f}$
 as a composition $\Tilde{f}=i\circ p$ where $p$ is a finite composition of $\mathcal{O}$-foldings and $i$ an $\mathcal{O}$-permutation. The only fold that can take place is between the initial and terminal ends of $\widetilde{x}_j$. After performing the maximal sequence of available folds, the resulting graph $T/G$ has one vertex more and, since there are no other available folds, the resulting map should be a permutation which is a contradiction. Hence, there are no available folds which means that $\Tilde{f}(\widetilde{x}_j)=h\cdot \widetilde{x}_j^{\pm}$, for some $h\in G$. 
      \end{enumerate}
\end{proof}

In the special case of free groups the following lemmas, i.e. Lemma \ref{Cor3.2.2}, Lemma \ref{lem3.2.4} and Lemma \ref{lem3.4.2}, have been proved in \cite{Tits1}. In the case of free products the proofs are essentially the same as in \cite{Tits1}, based in \autocite[Corollary~3.2.2]{Tits1}, \autocite[Lemma~3.2.4]{Tits1} and \autocite[Lemma~3.4.2]{Tits1}, respectively, but we reproduce them here for completeness reasons.

\begin{lemma}\label{Cor3.2.2}
   Let $[\phi]\in\mathrm{Out}(G,\mathcal{G})$ be an outer automorphism of $G$ such that $[\phi]([G_i])\subseteq[G_i],\ i=1,2,\ldots,k$, $T\in\mathcal{O}$, and let $\Tilde{f}:T\to T$ be a topological representative of $[\phi]$. If $H_i$ is a stratum consisting only of the orbit $G\cdot\widetilde{e}$ of an edge $\widetilde{e}\in E(T)$, then the reduced image $\Tilde{f}(\widetilde{e})$ crosses the orbit of $\widetilde{e}$ in either direction, at most once.
\end{lemma}
\begin{proof}
We may assume that $H_i$ is the highest stratum, i.e. $T_i=T$, and we distinguish two cases. 
\begin{enumerate}
    \item If $T_{i-1}/G$ is connected, then since $H_i=G\cdot\widetilde{e}$, either both of the endpoints of $e$ belong to $T_{i-1}/G$ or $e$ has one of its ends in $(T_i/G)-(T_{i-1}/G)$, where $e$ is the edge of $T/G$ corresponding to $G\cdot\widetilde{e}$.
    \begin{enumerate}
        \item Suppose that both of the endpoints of $e$ belong to $T_{i-1}/G$. Every maximal tree of the quotient graph $T_{i-1}/G$ contains the endpoints of the edge $e$ corresponding to $G\cdot\widetilde{e}$. Since $H_i$ consists only of the orbit $G\cdot\widetilde{e}$, every maximal tree of $T_{i-1}/G$ is also a maximal tree of $T_{i}/G=T/G$. Hence $\pi_1(T/G,X)\cong\pi_1(T_{i-1}/G,X)\ast\langle x_j\rangle$ where $X$ is a maximal tree of $T/G$ as above and the group element $x_j$ corresponds to the edge $e$. Since the $G$-invariant tree $T_{i-1}$ is also $\Tilde{f}$ invariant, Lemma \ref{mainlemma} (2) applies. 
        \item 
        Suppose now that $e$ has one of its ends, $v_i$, in $(T_i/G)-(T_{i-1}/G)$. The vertex $v_i$ must be non-free i.e. $G_{\widetilde{v}_i}\not=1$, being terminal. The image $\Tilde{f}(\widetilde{e})$ is determined by the images of its ends $\widetilde{w}$ and $\widetilde{v_i}$, since
        \[
        \Tilde{f}(\widetilde{e})=\left[\Tilde{f}(\widetilde{w}),\Tilde{f}(\widetilde{v_i})\right]=[\underbrace{\Tilde{f}(\widetilde{w})}_{\in T_{i-1}},\underbrace{x\widetilde{v}_i}_{\in H_i}]
        \]
        where $\phi(G_{\widetilde{v}_i})=G_{x\cdot \widetilde{v}_i}=xG_{\widetilde{v}_i}x^{-1}$. The group $G$ has a free product decomposition $G=\pi_1(T/G)\cong \pi_1(T_{i-1}/G)\ast G_{\widetilde{v}_i}$ such that $\phi\left(\pi_1(T_{i-1}/G)\right)\subseteq \pi_1(T_{i-1}/G)$. From Lemma \ref{mainlemma} (1) it follows that $x$ contains no element from $G_{\widetilde{v}_i}$ hence $\Tilde{f}(\widetilde{e})$ crosses the orbit $G\cdot\widetilde{e}^{\pm}$ at most (exactly) once.
    \end{enumerate}
     \item If $T_{i-1}/G$ is disconnected, denote by $C_1=c_1/G$ and $C_2=c_2/G$ (where $c_1,c_2\subseteq T_{i-1}$) its components. 
     \begin{enumerate}
         \item Suppose at first that $\Tilde{f}(c_i)\subseteq G\cdot c_i,\ i=1,2$. We will define the maps $\Tilde{h}_1,\Tilde{h}_2:T\to T$ whose supports are contained in $c_1,c_2$, respectively. If $C_j$ has a single vertex, let $\Tilde{h}_j:T\to T$ be the identity. If $C_j$ has at least two
vertices, choose one, $v_j$,  that is not an endpoint of $e$ and let $\Tilde{h}_j:T\to T$ be
a map with support in $c_j$ that is homotopic to the identity and that satisfies $(\Tilde{h}_j\Tilde{f})(\widetilde{v}_j)=\widetilde{v}_j$. Let $\Tilde{f}':T\to T$ be the topological representative defined by $\Tilde{f}'([u,w])=\left[(\Tilde{h}_1\Tilde{h}_2\Tilde{f})(u),(\Tilde{h}_1\Tilde{h}_2\Tilde{f})(w)\right]$ for every $u,w\in V(T)$. Since $\Tilde{f}(\widetilde{e})$ and $\Tilde{f}'(\widetilde{e})$ cross $G\cdot\widetilde{e}^{\pm}$ the same number of times, we may replace $\Tilde{f}$ with $\Tilde{f}'$.  In particular, we may assume
that $\Tilde{f}$ fixes the orbits of $\widetilde{v}_1$ and  $\widetilde{v}_2$. We add an edge $y$ with endpoints $v_1$ and $v_2$ and extend $\Tilde{f}$ by defining $\Tilde{f}(\widetilde{y})$ to be the geodesic $[\Tilde{f}(\widetilde{v}_1),\Tilde{f}(\widetilde{v}_2)]\in G\cdot\widetilde{y}$. This defines a topological representative (of a different
outer automorphism) to which the previous argument applies. 
\item If $\Tilde{f}(c_1)\subseteq G\cdot c_2$ and $\Tilde{f}(c_2)\subseteq G\cdot c_1$ then, arguing as above, we may assume that
there are vertices $x_j\in C_j$ such that $\Tilde{f}(\widetilde{x}_1)\in G\cdot\widetilde{x}_2$ and $\Tilde{f}(\widetilde{x}_2)\in G\cdot\widetilde{x}_1$. We add an edge $y$ as above and extend $\Tilde{f}$ by $\Tilde{f}(\widetilde{y})=[\Tilde{f}(\widetilde{x}_2),\Tilde{f}(\widetilde{x}_1)]\in G\cdot\widetilde{\overline{y}}$. The proof concludes as in the previous case.
     \end{enumerate}
     
\end{enumerate}    
\end{proof}

The following lemma enables us to establish a pairing between elements of $ \mathcal{L}([\phi])$ and $\mathcal{L}([\phi]^{-1})$.
\begin{lemma}[{{\cite[Lemma 4.11]{LymanCTs}}}]
    There is a unique free factor $F$ of minimal rank whose conjugacy class $[F]$ carries every line in $\Lambda^+$.
\end{lemma}
Note that $[F]$ contains at least one non peripheral element since it carries every line in $\Lambda^+$. We denote $[F]$ by $F(\Lambda^+)$. The rank of $F(\Lambda^+)$ is defined to be the rank of $F$.\par
We also need the following generalisations of two technical lemmas from \cite{Tits1}.
\begin{lemma}\label{lem3.2.4}
   For each $\Lambda^+\in \mathcal{L}([\phi])$ there is a unique $\Lambda^-\in \mathcal{L}([\phi]^{-1})$ such that $F(\Lambda^+)=F(\Lambda^-)$.
\end{lemma}
\begin{proof}
     We induct on the Kurosh rank $k+p$ of $G=G_1\ast\ldots\ast G_k\ast F_p$. If $k+p=1$, then there are no
exponentially-growing strata in any relative train track map representing an
iterate of $\phi$ so $\mathcal{L}([\phi])$ is empty.\par
Since $\mathcal{L}([\phi])$ is finite and $[\phi]$-invariant, by replacing $[\phi]$ by an iterate, we may assume that each element of $\mathcal{L}([\phi])$ is $[\phi]$-invariant. Choose $\Lambda^+\in \mathcal{L}([\phi])$. Since $F(\Lambda^+)$ is unique, then it is $[\phi]$-invariant. If the rank of $\Lambda^+$ is less than $k+p$, then  the
inductive hypothesis, applied to the restriction of $[\phi]$ to $F(\Lambda^+)$  implies that there exists a unique $\Lambda^-\in\mathcal{L}([\phi]^{-1})$ such that $F(\Lambda^+)=F(\Lambda^-)$. We may
therefore assume that there is a pairing between the elements of $\mathcal{L}([\phi])$ with rank less than $k+p$ and the the elements of $\mathcal{L}([\phi]^{-1})$ with rank less than $k+p$.\par
For any exponentially-growing stratum $(T_r/G)-(T_{r-1}/G)$ 
there are bi-infinite paths in $T_r/G$ crossing edges in $(T_r/G)-(T_{r-1}/G)$. It follows that the rank of each component of $T_{r-1}/G$ is
less than the rank of $T_r/G$.  In particular, as any two distinct laminations can not have non trivial common segments, there is at most one element of $\mathcal{L}([\phi])$ or $\mathcal{L}([\phi]^{-1})$  with rank $k+p$. It therefore suffices to assume that there is an element $\Lambda^+\in\mathcal{L}([\phi])$ with rank $k+p$ and to prove that there is an element $\Lambda^-\in\mathcal{L}([\phi]^{-1})$ with rank $k+p$.\par 
It follows from \autocite[Proposition~1.3]{LymanCTs} that after replacing $[\phi]$ by a further iterate if necessary, there is a reduced relative train track map $\Tilde{f}:T\to T$ and a filtration $\emptyset=T_0\subseteq T_1\subseteq\ldots\subseteq T_K=T$ representing $[\phi]$. We may also assume, \autocite[Lemma~4.6]{LymanCTs} that, after replacing $[\phi]$ by a further iterate if necessary,  $\Tilde{f}:T\to T$ is an eg-aperiodic relative train track map representing $[\phi]$. Now $\Lambda^+$ is associated to the top stratum $(T_K/G)-(T_{K-1}/G)$ and each element of $\mathcal{L}([\phi])$ and $\mathcal{L}([\phi]^{-1})$ with rank less than $k+p$ is carried by $\mathcal{F}=\mathcal{F}(T_{K-1}/G)$ which is the minimal free factor which carries $T_{K-1}/G$. Lemma \ref{Cor3.2.2}  implies that $(T_K/G)-(T_{K-1}/G)$ is neither a single edge, nor an edge path (unless it contains at least two non free vertices of which at least one does not belong in $T_{K-1}/G$).
It
follows that $\mathcal{F}$ is neither a single (conjugacy class of a) free factor of rank $k+p-1$ nor a pair of (conjugacy classes of) free factors with rank sum equal to $k+p$.\par  
Since $[\phi]$ and $[\phi]^{-1}$ have the same invariant free factors,
we may choose a relative train track map $\Tilde{f}':T\to T$ representing an iterate of $\phi^{-1}$ such that $\mathcal{F}$ is realized by a filtration element. The
transition submatrix for a non-exponentially-growing stratum is a permutation and so has an iterate that equals the identity. We may therefore assume,
after replacing $\Tilde{f}'$ by an iterate and enlarging the filtration if necessary, that
each non-exponentially-growing stratum $(T_i/G)-(T_{i-1}/G)$ is a single edge. Recall that $\Tilde{f}$ is eg-aperiodic. If $(T'_{K'}/G)-(T'_{K'-1}/G)$ is the topmost stratum then,
since $\Tilde{f}:T\to T$ is reduced, $\mathcal{F}=\mathcal{F}(T'_{K'-1}/G)$.  Since  $(T'_{K'}/G)-(T'_{K'-1}/G)$ cannot be a zero
stratum, the concluding observation of the preceding paragraph rules out the
possibility that it is a single edge. Thus $(T'_{K'}/G)-(T'_{K'-1}/G)$ is exponentially growing. Since
the expanding lamination associated to  $(T'_{K'}/G)-(T'_{K'-1}/G)$ is not carried by $\mathcal{F}'$ it must have rank $k+p$.
 \end{proof}
The laminations $\Lambda^+,\Lambda^-$ of the above lemma will be called \textbf{paired}, as in the case of automorhpisms of free groups.
\begin{lemma}\label{lem3.4.2}
    Suppose that:
    \begin{enumerate}
        \item $\Lambda^+\in \mathcal{L}([\phi])$ and $\Lambda^-\in\mathcal{L}([\phi]^{-1})$ are paired and $[\phi]$-invariant. 
        \item $\Gamma^+\in \mathcal{L}([\psi])$ and $\Gamma^-\in\mathcal{L}([\psi]^{-1})$ are paired and $[\psi]$-invariant. 
        \item Generic lines in $\Lambda^+,\Lambda^-$ are weakly attracted to $\Gamma^+$ (respectively $\Gamma^-$) under the action of $[\psi]$ (respectively $[\psi]^{-1}$).
        \item Generic lines in $\Gamma^+,\Gamma^-$ are weakly attracted to $\Lambda^+$ (respectively $\Lambda^-$) under the action of $[\phi]$ (respectively $[\phi]^{-1}$).
    \end{enumerate}
    Then $[\phi]^N$ and $[\psi]^N$ generate a free subgroup of rank two for all sufficiently large $N$.
\end{lemma}
\begin{proof}
    Let $\lambda^{\pm}$ and $\gamma^{\pm}$ be generic lines for $\Lambda^{\pm}$ and $\Gamma^{\pm}$ respectively. Let, also, $U^{\pm}$ and $V^{\pm}$ be
attracting neighborhoods of $\lambda^{\pm}$ and $\gamma^{\pm}$ respectively. \par There exists $k\ge0$ such that $[\phi]^k_{\#}(\gamma^+)\subseteq U^+$. Since $\{[\psi]^l_{\#}(V^+):\,l\ge0\}$ is a neighborhood basis for $\gamma^+$, there exists $l\ge0$ such that $[\phi]^k_{\#}(\{[\psi]^l_{\#}(V^+))\subseteq U^+$. This
inclusion remains valid if $k$ and/or $l$ are increased.\par 
Repeating this argument on various combinations of $[\psi]^{\pm1}$ and $[\psi]^{\pm1}$, we see
that for all sufficiently large $K$ and $L$ we have
\begin{itemize}
    \item $[\phi]^K_{\#}([\psi]^L_{\#}(V^+)),[\phi]^K_{\#}([\psi]^{-L}_{\#}(V^-))\subseteq U^+$,
    \item $[\phi]^{-K}_{\#}([\psi]^L_{\#}(V^+)),[\phi]^{-K}_{\#}([\psi]^{-L}_{\#}(V^-))\subseteq U^-$,
    \item $[\psi]^K_{\#}([\phi]^L_{\#}(U^+)),[\phi]^K_{\#}([\psi]^{-L}_{\#}(U^-))\subseteq V^+$,
     \item $[\psi]^{-K}_{\#}([\phi]^L_{\#}(U^+)),[\phi]^{-K}_{\#}([\psi]^{-L}_{\#}(U^-))\subseteq V^-$.
\end{itemize}
Defining $U_L^+$ (respectively $U_L^-,V_L^+,V_L^-$) by $[\phi]^L_{\#}(U^+)$  (respectively $[\phi]^L_{\#}(U_L^-),[\psi]_{\#}^L(V^+),[\psi]_{\#}^{-L}(V^-)$)  and defining $N=K+L$, we have
\begin{enumerate}
    \item $[\phi]^N_{\#}(V_L^+),[\phi]^N_{\#}(V_L^-)\subseteq U_L^+$,
    \item  $[\phi]^{-N}_{\#}(V_L^+),[\phi]^{-N}_{\#}(V_L^-)\subseteq U_L^-$,
    \item $[\psi]^N_{\#}(U_L^+),[\psi]^N_{\#}(U_L^-)\subseteq V_L^+$,
    \item $[\psi]^{-N}_{\#}(U_L^+),[\psi]^{-N}_{\#}(U_L^-)\subseteq V_L^-$.
\end{enumerate}
Since $U^{\pm}$ and $V^{\pm}$ are attracting neighborhoods, we also have
\begin{enumerate}
    \item[5.] $[\phi]^N_{\#}(U_L^+)\subseteq U_L^+, [\phi]^{-N}_{\#}(U_L^-)\subseteq U_L^-, [\psi]^N_{\#}(V_L^+)\subseteq V_L^+,  [\psi]^{-N}_{\#}(V_L^-)\subseteq V_L^-$. 
\end{enumerate}
Choose a hyperbolic element $g\in G$ 
 and $\delta=[x,gx]$, a fundamental domain of the action of $g$ on $T$, that is weakly attracted to $\Lambda^+$ under the action
of $[\phi]$ and to $\Lambda^-$ under the action of $[\phi]^{-1}$.
By $(3)$ and $(4)$, $[\phi]_{\#}^m(\delta)$ is weakly attracted to $\Gamma^+$ (respectively $\Gamma^-$)  under the action of $[\psi]$ (respectively $[\psi]^{-1}$) for all
sufficiently large $m$, say $m\ge M$. Let $x=[\phi]_{\#}^M(\delta)\in \mathcal{B}$. 
Since  
generic lines in attracting laminations cannot be periodic lines we can choose $L$ so that $x\not\in (U_L^{+}\cup U_L^{-}\cup V_L^{+}\cup V_L^{-})$. 
For large $K$ and hence large $N$, $[\phi]_{\#}^N(x)\in U_L^+$, $[\phi]_{\#}^{-N}(x)\in U_L^-$, $[\psi]_{\#}^N(x)\in V_L^+$, $[\psi]_{\#}^{-N}(x)\in V_L^-$. \par
Thus all hypotheses of \autocite[Lemma~3.4.1]{Tits1} are verified.

\end{proof}

\begin{theorem}\label{main1}
Let $[\phi]\in\mathrm{Out}(G,\mathcal{G})$ be an outer automorphism of exponential growth and let $\Lambda^+\in \mathcal{L}([\phi])$ and  $\Lambda^-\in \mathcal{L}([\phi]^{-1})$ be two paired
and $[\phi]$-invariant laminations. Suppose that $H$ is a subgroup of $\mathrm{Out}(G,\mathcal{G})$ containing $[\phi]$ and that there
is an element $[\psi]\in H$ of exponential growth such that generic lines of the four laminations $[\psi]^{\pm1}(\Lambda^{\pm})$ are weakly attracted to $\Lambda^+$ under the action of $[\phi]$ and to $\Lambda^-$ under the action of $[\phi]^{-1}$. Then $H$ contains a free subgroup of rank two. 
\end{theorem}
\begin{proof}
    Define $\Phi=[\psi][\phi][\psi]^{-1}$ and note that $\Gamma^+=[\psi]_{\#}(\Lambda^+)\in\mathcal{L}(\Phi^+)$ and $\Gamma^-=[\psi]_{\#}(\Lambda^-)\in \mathcal{L}(\Phi^{-1})$ are paired and $\Phi$-invariant.  Our hypothesis on the generic lines of $[\psi]_{\#}^{\pm1}(\Lambda^{\pm})$ can be restated as:
    \begin{itemize}
        \item $\Gamma^{\pm}$ is weakly attracted to $\Lambda^+$ (respectively $\Lambda^-$) under the action of $[\phi]$ (respectively $[\phi]^{-1}$).
        \item $[\psi]^{-1}_{\#}(\Lambda^{\pm})$ is weakly attracted to $[\psi]^{-1}_{\#}(\Gamma^+)$ (respectively $[\psi]^{-1}_{\#}(\Gamma^-)$) under the action of $[\psi]^{-1} \Phi [\psi]$ (respectively $[\psi]^{-1}\Phi^{-1}[\psi]$). 
    \end{itemize}
    This last condition is equivalent to:
\begin{itemize}
    \item $\Lambda^{\pm}$ is weakly attracted to $\Gamma^+$ (respectively $\Gamma^-$)  under the action of $\Phi$ (respectively $\Phi^{-1}$).
\end{itemize}
The result now follows from Lemma \ref{lem3.4.2}.
\end{proof}

\section{Subgroups carrying laminations}\label{subgroupssection4}
We fix a group $G$, a finite (non-trivial) free product decomposition of $G=G_1\ast\ldots\ast G_k\ast F_p$, the corresponding relative outer space $\mathcal{O}=\mathcal{O}(\mathcal{G})$ and an iwip automorphism $[\phi]\in\mathrm{Out}(G,\mathcal{G})$.

\begin{proposition}[{{\cite[Lemma 3.9, point (c)]{loxodromic}}}, {{\cite[Lemma 5.3.1]{Francaviglia2021OnTA}}}]\label{finiteindex} 
If $A$ is a subgroup of $G$ of finite Kurosh rank which carries $\Lambda_{[\phi]}^+$ and $A\cap xG_ix^{-1}=\{1\}$ or $xG_ix^{-1}$ for every elliptic free factor $G_i$ of $G$, then $A$ has finite index in $G$.   
\end{proposition}

In the special case of free groups the following lemma has been proved in \autocite[Proposition~2.6]{Lam}. We give a proof in the more general case of free products for reasons of completeness. 
\begin{proposition}\label{prp2.6bestvina}
Let $[\phi]\in \mathrm{Out}(G)$ be an iwip outer automorphism and $\Lambda=\Lambda_{[\phi]}^+$ its associated stable lamination. Let also $[\psi]\in \mathrm{Stab}(\Lambda)$ and $\Tilde{h}:T\to T$ be a relative
train-track representative of $[\psi]$. Then:
\begin{enumerate}
    \item If $T_0\subseteq T$ is a proper, $\Tilde{h}$-invariant $G$-subtree that is a union of strata and such that $T_0/G$ has no valence one free vertices, then there is $k\in \mathbb{N}$ such that $\Tilde{h}^k:T_0\to T_0$ is an isometry.
    \item If $\Tilde{h}$ has no exponentially growing strata, then there is $k\in \mathbb{N}$ such that $\Tilde{h}^k$ is an isometry.
\end{enumerate}
\end{proposition}

\begin{proof}
\begin{enumerate}
    \item Let $\ell$ be a line of $\Lambda$. By Proposition \ref{finiteindex} $\ell$ is a concatenation of nondegenerate segments in $T_0$ and $T\setminus T_0$. Notice that all of the segments in $T_0$ are Nielsen preperiodic or otherwise some $\Tilde{h}$-iteration would produce an arbitrarily long line segment contained in $T_0$ contradicting irreducibility of $[\phi]$ since $\ell$ is quasiperiodic. Again by quasiperiodicity, there is an upper bound to the length of both $T_0$ and $T\setminus T_0$-segments and hence only finitely many projectively inequivalent segments occur. Let $X$ be the disjoint union of the $G$-orbits of both $T_0$ and $T\setminus T_0$-line segments and $X\to T$ the natural immersion. We identify $G$-equivariantly two endpoints of line segments in $X$ if their images under the above immersion coincide in $T/G$. 
     Then we fold to convert the resulting map to an immersion $X'\to T$. If $v$ is a vertex of $X'$ and there are $g_1,g_2\in G$ such that $g_1g_2^{-1}\in G_{v}$ and $v=i(g_1e)=\tau(g_2e)$ then, from $X'$ we get an immersion $X''\to T$ where $X''$ contains the full (i.e. some conjugates of the $G_i$'s) stabilisers of vertices whenever the corresponding  stabiliser in $X'$ is not trivial. Equivalently, $X'/G$ is not necessarily a graph of groups whereas $X''/G$, by construction, is a graph of groups since we have filled all of the non trivial vertex "groups". In other words, $X'/G$ does not necessarily correspond to a free factor in $\mathcal{G}$ but $X''/G$ corresponds to the minimal free factor in $\mathcal{G}$ that carries the lamination. Since $\ell$ lifts to $X''$ by construction, the immersion $\pi:X''\to T$ must be, from Lemma \ref{finiteindex}, a finite-sheeted covering space. Therefore, a power of any hyperbolic element in $T_0$, say $\alpha$, lifts to $X''$. 
     It follows that $\alpha$ splits into finitely many subpaths each of which belongs to $\ell$ and therefore the set $\{\text{length}([\Tilde{h}^k(\alpha)]),\ k\in\mathbb{N}\}$ is finite (uniformly bounded by some $M>0$). As we may choose the initial hyperbolic element to be cyclically reduced, the same set is finite for any hyperbolic element. It is easy to see that the bottom stratum contained in $T_0$ has trivial corresponding submatrix in some power, since the restriction of the representative is irreducible without growth. It follows by induction that the transition submatrix corresponding to $T_0$ has a trivial power as otherwise we would have a hyperbolic element whose length after $\Tilde{h}$-iteration grows to infinity. 
    
    \item Let $T_0\subseteq T_1\subseteq \ldots \subseteq T_m=S$ be the filtration corresponding to the relative train
track map $\Tilde{h}$. Let’s consider the maximal $\Tilde{h}$-invariant sub-forest $S'=T_1\cup\ldots\cup T_{m-1}$.  By applying $1$, we can assume that some power of $\Tilde{h}$ acts as an isometry on $S'$.\par
By applying again $1$, we can assume that, after possibly changing $[\psi]$ with some power $[\psi]^k$, there is a relative train track representative, $\Tilde{h}:S\to S$ and a maximal
proper $\Tilde{h}$-invariant $G$-subgraph $S_0$ of $S$ (we denote by $H_0$ the quotient $S_0/G$)  such that
the restriction of $\Tilde{h}$ on $S_0$ is an isometry and, moreover, the induced map on the
quotient $H_0$, is the identity. In other words, we assume that $\Tilde{h}$ sends every edge of $S_0$ to
an edge of the same orbit or, equivalently, the transition matrix of $\Tilde{h}$ is the identity when
restricted to $S_0$.\par
Now we focus on the top stratum $S_1=\overline{S-S_0}$.  By assumption, the top stratum is non-exponentially growing. Therefore, we can assume that, it contains a single edge $\Tilde{e}$ and $\Tilde{h}$ can be chosen so that $\Tilde{h}(\Tilde{e})=\Tilde{e}\Tilde{a}$, where $\Tilde{a}$ a is some segment of $S_0$ corresponding to a
simple loop (we can do this without lost after changing $[\psi]$ with a power, as if it had the
form $\Tilde{h}(\Tilde{e})=\overline{\Tilde{e}}\Tilde{a}$, we could change $\Tilde{h}$ with $\Tilde{h}^2$). As $\Tilde{h}$ is a relative train track and $[\psi]^k\in\mathrm{Stab}(\Lambda)$, we have that for every positive integer $m$ which is multiple of $k$, $\Tilde{h}^m$-images of $\Tilde{e}$ produce
arbitrarily long segments of the lamination that are contained in $S_0$. More specifically,
we have for every $m=kd>0$,
\[
[h^m(\Tilde{e})]=h^m(\Tilde{e})=\Tilde{e}\Tilde{a}\Tilde{h}(\Tilde{a})\ldots \Tilde{h}^{m-1}(\Tilde{a})
\]
and its length is equal to $\mathrm{length}(\Tilde{e})+m\cdot \mathrm{length}(\Tilde{a})$ (as $\Tilde{h}$ acts as an isometry).\par
This contradicts quasiperiodicity in the case where this path does occur as a line segment,
and so we must have cancellation of the $S_0$-part (as the lamination would contain arbitrarily long segments of the proper sub-graph $S_0$). This forces the lines of the lamination
to have a very special form. In particular, in any line of the lamination in $S$-coordinates,
it is not possible to have two consecutive appearances of $\Tilde{e}$ without an appearance of $\overline{\Tilde{e}}$ in-between.\par
Therefore, a line of the lamination in $S$-coordinates is forced to be of the form:
\[
\ldots \Tilde{e}\cdot\widetilde{c_{-1}}\cdot\Tilde{e}\cdot\widetilde{x_{-1}}\cdot\Tilde{e}\cdot\widetilde{c_0}\cdot\overline{\Tilde{e}}\cdot\widetilde{x_0}\cdot\Tilde{e}\cdot\widetilde{c_1}\cdot\overline{\Tilde{e}}\cdot\widetilde{x_1}\cdot\Tilde{e}\cdot\widetilde{c_2}\cdot\overline{\Tilde{e}}\ldots
\]
for some segments $\widetilde{c_i}$'s which are concatenations of paths projecting to the same path as $\Tilde{a}$ to $H_0$ and some segments $\widetilde{x_i}$ which are contained in $H_0$.\par
There are two cases; either the projection of $\Tilde{e}$ to $S/G$ is a separating edge or not. Since $\Tilde{h}$ has non exponential growth and the induced map on the
quotient $H_0$ is the identity, we have already seen that, the induced map on $e$ is also the identity. In the non separating case, the lamination is  carried by the subgroup which is the fundamental group of the graph of groups
consisting of the disjoint union of two graph of groups, one corresponding to the simple loop $a$ (together with all the vertex stabilisers that $a$ crosses in $H_0$) and one corresponding
to the whole $H_0$, connected by the edge $e$. By Proposition \ref{finiteindex} this graph of groups must
correspond to a finite index subgroup, which is not possible as $e$ has just one lift, while
all edges crossing $a$ have two lifts (they should have the same number of lifts, provided
the assumption on the finite index). This leads us to a contradiction, except if $a$ crosses no edges, i.e. it is trivial.\par
Therefore, we have that $\Tilde{h}$ fixes the orbit of $\Tilde{e}$, inducing  the identity on $S/G-H_0$ and so on $S/G$.

\end{enumerate}    
\end{proof}

\section{The stretching homomorphism}\label{stretchingsection5}
In this section, we consider the kernel of the stretching homomorphism, whose existence is ensured by Proposition \ref{sigma}. We recall that we have fixed a group $G$, a finite (non-trivial) free product decomposition of $G=G_1\ast\ldots\ast G_k\ast F_p$, the corresponding relative outer space $\mathcal{O}=\mathcal{O}(\mathcal{G})$ and an iwip automorphism $[\phi]\in\mathrm{Out}(G,\mathcal{G})$. We denote by $\Lambda=\Lambda_{[\phi]}^+$ the associated stable lamination.

\begin{proposition}[{{\cite[Proposition 4.13 and 4.14]{LymanCTs}}}]\label{sigma}
Let $[\psi]\in\mathrm{Stab}(\Lambda)$ and $\Tilde{h}:T\to T$ an $\mathcal{O}$-representative of $[\psi]$. There exists a positive constant $\lambda=\lambda([\psi])=\lambda(\Lambda)$, which does not depend on the representative $\Tilde{h}$, such that for every $\epsilon>0$ there is some $N>0$ such that if $L$ is a line segment of $\Lambda$ of length greater than $N$, then 
\[
\left|\dfrac{\mathrm{length }[\Tilde{h}(L)]}{\mathrm{length }L}-\lambda\right|<\epsilon.
\]
Furthermore, there is a well defined homomorphism $\mathrm{Stab}(\Lambda)\to\mathbb{R},\ [\psi]\mapsto \ln{\lambda(\Lambda)}$, which induces a homomorphism
\[
\sigma:\mathrm{Stab}(\Lambda)\to\mathbb{Z}.
\]
\end{proposition}

\subsection{The reducible case}
In this subsection we prove that an element of $\mathrm{ker}(\sigma)$ is either non-exponentially growing or has a positive iterate that is reducible.
\begin{proposition}\label{prp6.1}
If $[\psi]\in\mathrm{Stab}(\Lambda)$ is exponentially growing and there exists some $k\in\mathbb{N}$ such that $[\psi]^k$ is reducible, then $[\psi]\not\in \mathrm{ker}(\sigma)$.
\end{proposition}
\begin{proof}
    Let $\Tilde{h}:T\to T$ be a relative train track representative of $[\psi]$.\par
    First, we note that every stratum, except possibly the top one, is non-exponentially growing. Indeed, if otherwise, some $H_r$ was exponentially growing and $\Tilde{e}\in H_r$, then we would have that the
lengths of tightenings of $\Tilde{h}$-iterates of $\Tilde{e}$ are arbitarily long (by the train track properties)
while they are line segments (by the definition of the stabiliser of the lamination), which
would mean that arbitarily long segments are contained in some proper subgraph (since $\Tilde{h}(T_{r-1})\subseteq T_{r-1}$). This is impossible since the lamination contains the edges not in $T_r$.\par
Therefore, since $[\psi]$ is exponentially growing, we can suppose, after changing $\Tilde{h}$ with
some iterate, if it is necessary, that there exists some $G$-subgraph $H_0$ which is a union of
strata, so that each of them is non-exponentially growing and, even more, by Proposition \ref{prp2.6bestvina}, $\Tilde{h}$ restricted to $H_0$ induces the identity in the quotient graph (i.e. it induces an isometry in the level of the tree). We also suppose that the top stratum is exponentially growing. In this
case, for any line of the lamination and by using the subgraph-overgraph decomposition
of the line, it follows that the lengths of long line segment grow exponentially and, in
fact, the value of $\lambda([\psi])$ is the Perron-Frobenius eigenvalue corresponding to the unique
exponentially growing (top) stratum. It is worth to note here, that the previous argument
really depends on the fact that $H_0$ is not trivial (or in other words, the fact that $[\psi]^k$ is
reducible), as otherwise the subgraph-overgraph decomposition cannot be used. More precisely, under our assumptions, it suffices to consider a long line segment of the form 
\[
p=p_0q_1\ldots p_{k-1}q_np_k
\]
where $n>0$, all $p_i$'s and $q_j$'s for $i=0,\ldots,k$, $j=1,\ldots,n$ have uniformly bounded length and they are contained in $H_0$ and $T\setminus H_0$, respectively. By the relative train track properties, for every (multiple of $k$) $m>0$ we have that the path
\[
[h^m(p)]=h^m(p_0)[h^m(q_1)]h^m(p_1)\ldots h^m(p_{k-1})[h^m(q_n)]h^m(p_k)
\]
is again a line segment as the segments $h^m(p_0)$ and $h^m(p_k)$ can never be cancelled and the
previous argument for the growth can be applied.
\end{proof}

\subsection{The irreducible case}
In this subsection we prove that every element of $\mathrm{ker}(\sigma)$ is non-exponentially growing.
\begin{proposition}[{{\cite[Corollary 6.1.8]{Francaviglia2021OnTA}}}]\label{lem6.3}
    Let $\Tilde{h}:S\to S$ be a train track map representing some irreducible $[\psi]\in\mathrm{Out}(G,\mathcal{G})$. Then for every $C>0$ there is a positive number $M>0$ such that if $L$ is any
path, then one of the following holds:
\begin{enumerate}
    \item $[\Tilde{h}^M(L)]$ contains a legal segment of length more than $C$.
    \item $[\Tilde{h}^M(L)]$ has fewer illegal turns than $L$.
    \item $L$ is a concatenation $x\cdot y\cdot z$ such that $y$ is Nielsen preperiodic and each of $x,z$ has length less than or equal to $2C$ and at most one illegal turn.
\end{enumerate}
\end{proposition}

\begin{lemma}\label{lem6.7}
   Let $[\psi]\in\mathrm{Stab}(\Lambda)$ be a relative iwip and $\Tilde{h}:Q\to Q$ be a stable train track representative of $[\psi]$. If $\ell$ is some line of the lamination $\Lambda$ in $S$-coordinates,  then there is an integer $m$ such that
it is not possible for $\ell$ to contain a concatenation of $m$ Nielsen preperiodic subpaths. 
\end{lemma}

\begin{proof}
    Let $\Tilde{h}:Q\to Q$ be a stable train track representative of $[\psi]$ (\cite{CollinsTurner}, \cite{SykStableRepresentatives}). For a stable train track representative there is exactly one (orbit of) indivisible Nielsen periodic path $P$ in $S$. We suppose that there is no bound in the number
of concatenations of indivisible Nielsen paths in $\ell$. So by quasiperiodicity we have that every line segment
is contained in some concatenation of equivalent paths of the form $P_1P_2\ldots P_n$ (where every $P_i$ is in the orbit of $P$). This contradicts to Remark \ref{rem4.4}. Therefore, there is some bound on the number of $N$-periodic consecutive paths that a line
of the lamination could cross.  
\end{proof}

\begin{remark}\label{remof6.7}
    Lemma \ref{lem6.7} holds for any train track representative $\Tilde{g}:S\to S$ of $[\psi]$, with a possibly different $m$ that depends on the distance between $S$ and $Q$. This follows from the bounded cancellation lemma. In particular,  for any fixed neighborhood of $Q$ we can choose a uniform bound $m'$ that applies to all points of that neighborhood.   
\end{remark}

\begin{lemma}\label{lem6.5}
    Let $[\psi], [\psi]^{-1}\in\mathrm{Stab}(\Lambda)$ be two iwip automorphisms and $\Tilde{h}:S\to S$, $\Tilde{h}':S'\to S'$ be train track map representatives of $[\psi]$ and $[\psi]^{-1}$, respectively. If $\tau:S\to S'$ and $\tau':S'\to S$ are $\mathcal{O}$-maps, then for any $C>0$ there is a positive constant $N_0$ such that if $\ell$ is a line of the lamination $\Lambda$, and $\ell'$ is a line of $\Lambda$ in $S'$-coordinates, then for every $M\ge N_0$ at least  one of the following holds:
    \begin{enumerate}
        \item[(A)] $[\Tilde{h}^M(\ell)]$ contains a legal segment of length more than $C$.
        \item[(B)] $[\Tilde{h'}^M(\ell')]$ contains a legal segment of length more than $C$.
    \end{enumerate}
\end{lemma}
\begin{proof}
    By Remark \ref{remof6.7}, there is some integer $m$ such that it is not possible  to concatenate
more than $m$, Nielsen periodic paths both in $\ell$ and $\ell'$. In fact, the same is true, with the same $m$, in all positive iterates $\title{h}^k(\ell)$ as, for some stable train track representative at $Q\in\mathcal{O}$ of $[\psi]$ we have $d([\psi]^k(Q),[\psi]^k(S))=d(Q,S)$. Similarly for $(\Tilde{h}')^k(\ell')$.\par
 Without loss of generality, we may assume that the positive constant $C$ is larger than the critical constants of both $\Tilde{h}$ and $\Tilde{h}'$. Let $M$ be the maximum of the two integers given in Proposition \ref{lem6.3} applied to $\Tilde{h},C$ and $\Tilde{h}',C$. Let’s fix a sufficiently large integer $s=s(\Tilde{h},\Tilde{h}',\tau,\tau',M)$, that will be specified later. Suppose that $(A)$ does not hold for $N_0=sM$. We will apply Proposition \ref{lem6.3} only
to $\Tilde{h}^M$-admissible segments (i.e. segments $L\subseteq \ell$ such that $\Tilde{h}^M(\partial L)\subseteq [\Tilde{h}^M(\ell)]$).  By our
assumption, the assertion $1$ of Proposition \ref{lem6.3} does not hold. If we further restrict
to segments $L$ with more than $m+2$ illegal turns, then the assertion $3$ of Proposition \ref{lem6.3} does not hold either  (because of our assumption on $[\psi]$). So, for such segments the assertion $2$ of Proposition \ref{lem6.3}  is always true. We can represent $\ell$ as a concatenation of such segments
of uniformly bounded length where that uniform bound does not depend on $\ell$,  but only on $\Tilde{h},\Tilde{h}',\tau,\tau',M$ (since we will apply the same argument using $[\tau \Tilde{h}(\ell)]$, $\Tilde{h}'$ instead of $\ell,\Tilde{h}$ respectively).\par
Let $p$ be an upper bound to the number of illegal turns in each segment (there are
finitely many, since the segments have uniformly bounded length and there are finitely
many orbits of illegal turns for any $\mathcal{O}$-maps). Let’s now fix some number $a$ such that $\frac{p-1}{p}<a<1$. For a long enough segment $L$ in $\ell$, we have
\[
\dfrac{\text{number of illegal turns in } [\Tilde{h}^M(L)]}{\text{number of illegal turns in } L}<a,
\]
since the number of illegal turns in $L$  is less than $p$ and the number of illegal turns in $[\Tilde{h}^M(L)]$ is strictly less than the number of illegal turns in $L$.\par
By applying the same argument to $\Tilde{h}^M(\ell)$ and more generally to $\Tilde{h}^{kM}(\ell)$,  we can see
that for any given $s>0$  and for any long enough segment $L\subseteq \ell$ (with length depending
on $s$) we have that
\[
\dfrac{\text{number of illegal turns in }[\Tilde{h}^{sM}(L)]}{\text{number of illegal turns in }L}<a^s,
\]
as otherwise $(A)$ would hold with $N_0=sM$. Since any such legal segment has length
bounded above by $C$ and bounded below by the length of the shortest edge (with the
possible exception of the two segments containing the endpoints that could be arbitrarily small but get cancelled), the lengths can be
compared with two inequalities to the number of illegal turns. Therefore if $(A)$ fails, there exists a constant $A=A(\Tilde{h},C)$ with the property
\[
\dfrac{\mathrm{length } [\Tilde{h}^{sM}(L)]}{\mathrm{length } [L]}<Aa^s.
\]
Similarly, we can apply the same argument using $[\tau \Tilde{h}^{sM}\ell]$  in place of $\ell$ and with $\Tilde{h}'$ in
place of $\Tilde{h}$. If $(B)$  fails as well, (with $N_0=sM$), then we reach the similar inequality
\[
\dfrac{\mathrm{length } [(\Tilde{h}')^{sM}\tau \Tilde{h}^{sM}(L)]}{\mathrm{length } [\tau \Tilde{h}^{sM}(L)]}<Ba^s.
\]
for some $B$ depending only on $\Tilde{h}',C$. Since $(\Tilde{h}')^{sM}\tau \Tilde{h}^{sM}$, $\tau$ are $\mathcal{O}$-maps between the same tree, Lemma \ref{lem2.20} implies that they coincide for every sufficiently long path, except for some bounded error near endpoints. In particular, for long enough $L$, we have that the ratio of their lengths is bounded above by $2$ and below by $\frac{1}{2}$. Therefore, by multiplying the previous inequalities and by replacing $(\Tilde{h}')^{sM}\tau \Tilde{h}^{sM}$ with $\tau$, we get the inequality:
\[
\dfrac{\mathrm{length }[\tau(L)]}{\mathrm{length }[\tau\Tilde{h}^{sM}(L)]}\cdot \dfrac{\mathrm{length }[\Tilde{h}^{sM}(L)]}{\mathrm{length } [L]}<2ABa^{2s}.
\]
On the other hand, by applying again Lemma \ref{lem2.20} and the fact that $\tau'\tau$ and the identity are
both $\mathcal{O}$-maps between the same trees, we get:
\[
\dfrac{\mathrm{length }[\Tilde{h}^{sM}(L)]}{\mathrm{length }[\tau\Tilde{h}^{sM}L]}\cdot \dfrac{\mathrm{length } [\tau(L)]}{\mathrm{length } [L]}>\dfrac{1}{2\text{Lip}(\tau)\text{Lip}(\tau')}.
\]
By sending $s$ to infinity, the previous inequalities lead us to a contradiction, as by construction $a<1$. 
\end{proof}

\begin{proposition}\label{prp6.9}
    Let $\Lambda=\Lambda_{[\phi]}^+\in\mathcal{IL}$ and $[\psi]$ an iwip outer automorphism of $G$. Then
either the forward $[\psi]$-iterates of $\Lambda$ weakly converge to $\Lambda_{[\psi]}^+$ or $\Lambda=\Lambda_{[\psi]}^-$. In
particular, if $[\psi]\in \mathrm{Stab}(\Lambda)$, then either $\Lambda=\Lambda_{[\psi]}^+$ or $\Lambda=\Lambda_{[\psi]}^-$.
\end{proposition}
\begin{proof}
We will use the notation of Lemmas \ref{lem6.7} and \ref{lem6.5}. Let $\ell$ be a leaf of $\Lambda$ in $S$-coordinates. We apply the Lemma \ref{lem6.5} to $[\Tilde{h}^k(\ell)]$ for $K>0$ and some $C$ which is larger than both the critical constants of $\Tilde{h}$ and $\Tilde{h}'$. If for some $K>0$  the
statement $(A)$ holds, then it follows by quasiperiodicity that the forward iterates weakly converge to $\Lambda_{[\psi]}^+$ since we have that the length of reduced images tends to infinity. So
we have arbitrarily long legal segments and the properties of the lamination imply the convergence.\par
The remaining possibility is when $[\tau \Tilde{h}^k(\ell)]$ contains a $S'$ legal segment of length more than $C$ for all $k>0$. But this means that $[\tau(\ell)]$ which equals, up to bounded error, to $[\Tilde{h}'^k\tau \Tilde{h}^k(\ell)]$, contains an arbitrarily high $\Tilde{h}'$-iterate of a legal segment and, as before, quasiperiodicity
implies that $\Lambda=\Lambda_{[\psi]}^{-}$.
 \end{proof}

 \begin{corollary}\label{cor6.11}
 If $[\psi]\in\mathrm{Stab}(\Lambda)$ is exponentially growing, then $[\psi]\not\in \mathrm{ker}(\sigma)$.
 \end{corollary}
 \begin{proof}
For automorphisms with some reducible power, we have already proved it, in Proposition \ref{prp6.1}. For iwip automorphisms, it follows, by Proposition \ref{prp6.9}, that $\Lambda_{[\phi]}^+=\Lambda_{[\psi]}^+$ (after possibly changing $[\psi]$ with $[\psi]^{-1}$ if necessary). In this case, we can choose an optimal
train track representative $\Tilde{h}$ of $[\psi]$ for the topological representative of $[\psi]$ used for the
computation of $\lambda([\psi])$ (Proposition \ref{sigma}, {{\cite[Proposition 4.13 and Proposition 4.14]{LymanCTs}}}). Since $[\psi]$ is exponentially growing, the Perron-Frobenious eigenvalue $\lambda_{[\psi]}$ (which is equal to $\lambda([\psi])$) of the corresponding irreducible matrix $A(\Tilde{h})$ is greater than $1$. Therefore, the image of $[\psi]$ under $\sigma$, corresponding to $\ln{\lambda([\psi])}$, is positive.
 \end{proof}

\section{The second main theorem}\label{secondmainsection7}
We are now in position to prove our main technical result for the stabiliser of the lamination of an iwip outer automorphism $[\phi]\in\mathrm{Out}(G,\mathcal{G})$.
\begin{lemma}\label{lem7.1}
    Let $\Tilde{h}:S\to S$ be a relative train track representative of $[\psi]\in\mathrm{ker}(\sigma)$. Then there is some positive integer $k$ such that $\Tilde{h}^k$ induces the identity on $S/G$. In
particular, $[\psi]^k$ fixes $S$.
\end{lemma}
\begin{proof}
It follows immediately from Corollary \ref{cor6.11} and Proposition \ref{prp2.6bestvina} (part $2$).    
\end{proof}

\begin{theorem}\label{thm7.2}
     Let $[\phi]\in\mathrm{Out}(G,\mathcal{G})$ be an iwip outer automorphism and $\mathrm{Stab}(\Lambda)$ the  stabiliser of the associated stable lamination. Then $\mathrm{Stab}(\Lambda)$ is an infinite cyclic extension of a virtually
elliptic subgroup $A$. In other words, for every element $[\psi]$ of $A$ there is a positive integer $k$ and some element $S$ of $\mathcal{O}$ such that $[\psi]^k(S)=S$.
\end{theorem}
\begin{proof}
We consider the stretching homomorphism, $\sigma:\mathrm{Stab}(\Lambda)\to\mathbb{Z}$ given by Proposition \ref{sigma}. The quotient $\mathrm{Stab}(\Lambda)/\mathrm{ker}(\sigma)$  is isomorphic
to a subgroup of $\mathbb{Z}$. Also, it is not trivial, as $[\phi]\in \mathrm{Stab}(\Lambda)- \mathrm{ker}(\sigma)$, since, from Corollary \ref{cor6.11}, $[\phi]$ does not belong in $\mathrm{ker}(\sigma)$ being exponentially growing. The last part of the theorem
follows from Lemma \ref{lem7.1} by setting $A=\mathrm{ker}(\sigma)$, as every element of $\mathrm{ker}(\sigma)$ has an iterate which fixes a point of $\mathcal{O}$.   
\end{proof}

The next corollary generalises the well known result of Bestvina, Feighn and Handel which says that the stabiliser of the lamination of an iwip automorphism of a free group is virtually cyclic \autocite[Theorem~2.14]{Lam}. As we will see in the next section, in the context of free products, there are examples of iwip automorphisms with big stabilisers (see Remark \ref{gibcent}). However, under the extra assumption that all the elliptic free factors $G_i$ are finite, we achieve the exact analog of the above mentioned result. 
\begin{corollary}\label{cor7.3}
    Let $G=G_1\ast \ldots \ast G_k\ast F_p$, where each free factor $G_i$ is a finite group. Then for every iwip outer automorphism $[\phi]\in\mathrm{Out}(G,\mathcal{G})$, the group $\mathrm{Stab}(\Lambda_{[\phi]}^+)$   is virtually (infinite) cyclic. In particular, $\mathrm{Stab}(\Lambda_{[\phi]}^+)$  is finitely generated.
\end{corollary}
Note that under the hypothesis that every factor $G_i$ is finite, the decomposition $G=G_1\ast \ldots \ast G_k\ast F_p$ is the Grushko one.

\begin{proof}
    By Lemma \ref{lem7.1}, every element of $A=\mathrm{ker}(\sigma)$    (virtually) fixes a point. We will
prove that under our assumptions $A$ is finite. Therefore, $\mathrm{Stab}(\Lambda_{[\phi]}^+)/A\cong \mathrm{Im}(\sigma)\leq\mathbb{Z}$ hence $\mathrm{Stab}(\Lambda_{[\phi]}^+)$  is virtually cyclic.\par
For the finiteness of $A$, as we have already said, every element $[\psi]\in A$ virtually fixes a point. Since
every $G_i$ is finite, the stabilisers of points $T\in\mathcal{O}$ are finite as well (see as well \cite{Outspaceprod}). It follows that every element $[\psi]$ of the kernel $A$ has finite order which means that $A$ is a torsion group. In addition, $\mathrm{Out}(G)$  is virtually torsion free  {{\cite[Corollary 5.3]{Outspaceprod}}} since each $G_i$ is finite and thus $A$ must be finite.
\end{proof}

We denote by $\mathrm{Stab}_0(\Lambda_{[\phi]}^+)\coloneqq \mathrm{Out}_0(G,\mathcal{G})\cap\mathrm{Stab}(\Lambda_{[\phi]}^+)$  the finite index subgroup of $\mathrm{Stab}(\Lambda_{[\phi]}^+)$ which consists of  the automorphisms that preserve the conjugacy class of every $G_i$. In general, the restriction $[\psi]_i$ of any element $[\psi]\in \mathrm{Out}_0(G,\mathcal{G})$ to the free factor $G_i$ induces a well defined, up to an inner automorphism of $G_i$, element $[\psi]_i\in\mathrm{Aut}(G_i)$.  If we further assume that $[\psi]\in\mathrm{Out}(G,\{G_i\}^{(t)})$, i.e. $[\psi]\in \mathrm{Out}_0(G,\mathcal{G})$ acts as  as a conjugation by an element of $G$ on each $G_i$, then $[\psi]_i$ is an inner automorphism of the corresponding factor $G_i$.\par
Note that if $[\psi]\in\mathrm{Out}_0(G,\mathcal{G})$  acts as the identity on the graph of groups $\Gamma=T/G$ for some $T\in\mathcal{O}$, then it can be represented by a $G$-equivariant isometry $\Tilde{h}:T\to T$, preserving the orbit of every vertex and of every edge. More precisely, if
$\Tilde{v}$ is a non-free
vertex of $T$ such that $\Tilde{h}(\Tilde{v})=g\Tilde{v}$, $g\in G$,  then for any edge $\Tilde{e}$ emanating from $\Tilde{v}$, we have that $\Tilde{h}(\Tilde{e})=gd\Tilde{e}$ where $d\in G_i$ (whenever the stabiliser $G_{\Tilde{v}}$ is conjugate to $G_i$). In general, as $\Tilde{h}$ is an $\mathcal{O}$-map representing $[\psi]$, we have that $\Tilde{h}(g\Tilde{e})=\psi(g)\Tilde{h}(e)$. If we further assume that $[\psi]\in\mathrm{Out}(G,\{G_i\}^{(t)})$, then $\psi(d)=g_idg_i^{-1}$, $g_i,d\in G_i$ (as $[\psi]$ induces an inner automorphism on each $G_i$). Note also that whenever we choose a specific $\mathcal{O}$-map, we have essentially
chosen a specific representative of $\psi$ in its $\mathrm{Out}(G)$-class.\par
 For any $[\psi]\in\mathrm{ker}(\sigma)$, after possibly taking a suitable iterate, there is (Lemma \ref{lem7.1}) a tree $S\in \mathcal{O}$ such that $[\psi]\in\mathrm{Stab}_0(S)$, i.e. $[\psi]$
induces the identity on $S/G$.  In this case, there is an $\mathcal{O}$-map, $\Tilde{h}:S\to S$, representing $[\psi]$ where $\Tilde{h}$ has a very special form, as above. If $\ell$ is a line of
the lamination in $S$-coordinates, after possibly changing the topological representative $\Tilde{h}$ (equivalently, choosing a different representative of $\psi$ in its $\mathrm{Out}(G)$-class), we can assume
that $[\Tilde{h}(\ell)]=\Tilde{h}(\ell)=\ell$ (there is no cancellation for $\Tilde{h}$ as it is an immersion) and $\Tilde{h}(\Tilde{v})=\Tilde{v}$ for some non-free vertex $\Tilde{v}$ which is contained in $\ell$. This implies that $\Tilde{h}$  acts as the identity
on any line segment of $\ell$. More precisely, let’s suppose that $L$ is
a line segment of $\ell$ emanating from $\Tilde{v}$. If we suppose that the first edge of $L$ is $g_1\Tilde{e}$, for some $g_1\in G_{\Tilde{v}}$,  then for some fixed $g_e\in G_{\Tilde{v}}$ (depending only on the orbit of $\Tilde{e}$):
\[
\Tilde{h}(g_1\Tilde{e})=\psi(g_1)\Tilde{h}(\Tilde{e})=\psi(g_1)g_e\Tilde{e}.
\]
In this case we get the equation $\psi(g_1)g_e=g_1$, as edge stabilisers are trivial. Similarly, we
can continue in $\ell$ and we get  equations for any (orbit of) edge and for any different $g_i$ that we can "read" in the line $\ell$ of the lamination.\par
Now we will restrict our attention on automorphisms of $A_0\coloneqq \mathrm{ker}(\sigma)\cap \mathrm{Out}(G,\{G_i\}^{(t)})$. In this case, if $[\psi]\in A_0$, then the equation that we get from the lamination, as before, can
be written as $g_2g_1g_2^{-1}g_e=g_1$ for some $g_2\in G_{\Tilde{v}}$ (as $[\psi]$  induces an inner automorphism on
each $G_{\Tilde{v}}$). For example, if $G_{\Tilde{v}}$ is abelian, this equation immediately implies that $g_e=1$.
\begin{theorem}\label{thm7.4}
    Suppose that every $\mathrm{Inn}(G_i)$ is finite. Then the following hold:
    \begin{enumerate}
        \item There is a normal torsion subgroup $A_0$ of $\mathrm{Stab}_0(\Lambda_{[\phi]}^+)$,  such that the group $\mathrm{Stab}(\Lambda_{[\phi]}^+)/A_0$ has a normal subgroup $B/A_0$ so that 
        \[
        B/A_0\cong \bigoplus_{i=1}^k S_i\leq \bigoplus_{i=1}^k\mathrm{Out}(G_i), \text{ for some } S_i\leq \mathrm{Out}(G_i)
        \]
        and 
        \[
        \mathrm{Stab}(\Lambda_{[\phi]}^+)/B\cong \mathbb{Z}.
        \]
        \item If we further suppose that each $\mathrm{Out}(G_i)$ is virtually torsion free or that $G$ is finitely
generated (i.e every $G_i$ is finitely generated), then $ \mathrm{Stab}(\Lambda_{[\phi]}^+)$  has a (torsion free)
finite index subgroup $K$ such that:
\[
K/B'\cong \mathbb{Z}
\]
for some normal subgroup $B'$ of $K$ where
\[
B'\cong \bigoplus_{i=1}^k S_i', \text{ for some } S_i'\leq \mathrm{Out}(G_i).
\]
    \end{enumerate}
\end{theorem}
\begin{proof}
    Let’s denote by $B=\mathrm{ker}(\sigma)\cap \mathrm{Stab}_0(\Lambda_{[\phi]}^+)$ and by $A_0=B\cap \mathrm{Out}(G,\{G_i\}^{(t)})$. For any $[\psi]\in A_0$, after possibly changing it by some power, we can assume that the induced  automorphism $[\psi]_i$ on each $G_i$  is the identity (since $\mathrm{Inn}(G_i)$ is finite for every $i$). In
other words, there is an $\mathcal{O}$-map $\Tilde{h}:S\to S$ representing $[\psi]$, such that for any edge $\Tilde{e}$ in the
fundamental domain of $S$, we have that $\Tilde{h}(\Tilde{e})=g_e\Tilde{e}$. Now, exactly as in the discussion
above, since (the orbit of) every edge $\Tilde{e}$ appears in any line of the lamination, we have that
for every $\Tilde{e}$ we may get an equation of the form $\psi(g_1)g_e=g_1g_e=g_1$ which implies that $g_e=1$. In particular, $\Tilde{h}$ is the identity on the level of the tree, which means that $[\psi]$ is the identity.
It follows that $A_0$ is a torsion subgroup (i.e. every element of $A_0$ has finite order).\par
By considering the direct product over all $G_i$'s we get a (surjective) homomorphism
\[
F:\mathrm{Out}_0(G,\mathcal{G})\to\bigoplus_{i=1}^k\mathrm{Out}(G_i)
\]
with $\mathrm{ker}(F)=\mathrm{Out}(G,\{G_i\}^{(t)})$. We now restrict the natural map $F$ to the subgroup $B$. In this case, we
get the restriction homomorphism
\[
F_B:B\to \mathrm{Inn}(F_B)\leq \bigoplus_{i=1}^k\mathrm{Out}(G_i).
\]
By the first isomorphism theorem, we have that $B/\mathrm{ker}(F_B)$ is isomorphic to $\mathrm{Im}(F_B)\leq \bigoplus_{i=1}^k\mathrm{Out}(G_i)$. In
addition, $\mathrm{ker}(F_B)=B\cap \mathrm{Out}(G,\{G_i\}^{(t)})=A_0$. Then, the first part of the theorem follows, as $\mathrm{Stab}_0(\Lambda_{[\phi]}^+)/B\cong \mathbb{Z}$.\par
The second part is an easy application of the first part, since our assumptions that
each $G_i$ is finitely generated implies that $\mathrm{Out}(G)$  is virtually torsion free and hence the torsion
subgroup $A_0$ is finite. Indeed, since the quotient $G_i/Z(G_i)\cong \mathrm{Inn}(G_i)$  is finite for every $i$, it follows that every $G_i$ has an abelian subgroup $H_i$ of finite index, and therefore $\mathrm{Out}(G_i)$ is virtually torsion free. By \cite{Outspaceprod}, $\mathrm{Out}(G)$ is also virtually torsion free.\par
By considering a finite index torsion free subgroup $K$ of $\mathrm{Stab}(\Lambda_{[\phi]}^+)$, we have that $K\cap A_0=1$ and by setting $B'=K\cap B$  the result follows by the first part. 
\end{proof}

An immediate corollary of the previous theorem, is the following:

\begin{theorem}\label{thm7.5}
    Let $[\phi]\in\mathrm{Out}(G,\mathcal{G})$ be an iwip outer automorphism of $G$. Suppose that every $\mathrm{Aut}(G_i)$ is finite and either $G$ is finitely generated or $\mathrm{Out}(G)$ is virtually torsion free. Then $\mathrm{Stab}(\Lambda_{[\phi]}^+)$  is
virtually (infinite) cyclic.
\end{theorem}
\begin{proof}
    Since the groups $\mathrm{Out}(G_i)$ and $\mathrm{Inn}(G_i)$ are finite, the group $B'$ in Theorem \ref{thm7.4} is finite and it is contained in a
torsion free-group $K$, which means that $B'=1$.  As a consequence, $K$ is isomorphic to $\mathbb{Z}$. This implies that  $\mathrm{Stab}(\Lambda_{[\phi]}^+)$ has a finite index subgroup isomorphic to $\mathbb{Z}$, i.e. it is virtually
infinite cyclic. 
\end{proof}

\section{Applications}\label{appsection8}
We start with applications of our second main theorem regarding the relative centralisers of iwip automorphisms. Recall that the relative centraliser of $[\phi]\in\mathrm{Out}(G,\mathcal{G})$ is defined as $C([\phi])=\{[\psi]\in\mathrm{Out}(G,\mathcal{G}),\, [\psi]\cdot [\phi]=[\phi]\cdot [\psi]\}$.

\begin{theorem}\label{thm7.6}
    Let $[\phi]\in\mathrm{Out}(G,\mathcal{G})$ be an iwip outer automorphism of $G$. Then there is a normal virtually elliptic
subgroup $B$ of $C([\phi])$ such that $C([\phi])$ is a cyclic extension of $B$.
\end{theorem}
\begin{proof}
 This is an immediate application of Theorem \ref{thm7.2} as $C([\phi])\leq \mathrm{Stab}(\Lambda_{[\phi]}^+)$. The result follows
by considering the subgroup $B=C([\phi])\cap \mathrm{ker}(\sigma)$.   
\end{proof}

In \cite{francaviglia2020minimally} Francaviglia, Martino and the second named author, proved that for any iwip automorphism $[\phi]\in\mathrm{Out}(G,\mathcal{G})$ the set of points $M_{[\phi]}$ that admit train track representatives is locally finite. In fact, $M_{[\phi]}$ coincides with the set of minimally displaced points for $[\phi]$, under the
(asymmetric) Lipschitz metric, i.e. the points $S\in\mathcal{O}$ such that $d(S,[\phi](S))=\ln{\lambda_{[\phi]}}$. (For
the definitions and more details for the Lipschitz metric and the Min-Set $M_{[\phi]}$, see \cite{francaviglia2015stretching}).\par
Combining the locally finiteness of $M_{[\phi]}$ with Theorem \ref{thm7.2} and the
action of $C([\phi])$ on $M_{[\phi]}$, we get that any point of $M_{[\phi]}$ is virtually fixed by any element of $\mathrm{ker}(\sigma)\cap C([\phi])$. We state the main result of \cite{francaviglia2020minimally}.
\begin{theorem}[{{\cite[Corollary~6.10.]{francaviglia2020minimally}}}]\label{thm7.7}
    Let $[\phi]\in\mathrm{Out}(G,\mathcal{G})$ be an iwip outer automorhism. Then $M_{[\phi]}$ is (uniformly) locally finite.
\end{theorem}
We also list some well known properties of $M_{[\phi]}$ that we will need in the proof of the following theorem (for
details, see \cite{francaviglia2015stretching}).
\begin{itemize}
    \item $T\in M_{[\phi]}$ if and only if $d(T,[\phi](T))=\inf_{S}d(S,[\phi](S))=\ln{\lambda_{[\phi]}}$.
    \item $[\phi](T)\in M_{[\phi]}$, for every $T\in M_{[\phi]}$.
    \item If $T\in M_{[\phi]}$, $[\psi]\in C([\phi])$, then $[\psi](T)\in M_{[\phi]}$, since
    \[
    d([\psi](T),[\phi]([\psi](T)))=d(T,[\psi]^{-1}[\phi][\psi](T))=d(T,[\phi](T))=\ln{\lambda_{[\phi]}}.
    \]
\end{itemize}

\begin{theorem}\label{thm7.8}
    Let $[\phi]\in\mathrm{Out}(G,\mathcal{G})$ be an iwip outer automorhism. Then $C([\phi])$ has a
normal subgroup $B'$ such that $C([\phi])/B'$ is isomorphic to $\mathbb{Z}$. Moreover, for every $T\in M_{[\phi]}$ and for every $[\psi]\in B'$, there is some $k\in\mathbb{Z}$ such that $[\psi]^k(T)=T$.
\end{theorem}
\begin{proof}
    Let’s consider the subgroup $B'=\mathrm{ker}(\sigma)\cap C([\phi])$. By Theorem \ref{thm7.6}, any element of $B'$ virtually fixes some point $S\in \mathcal{O}$. Let $T\in M_{[\phi]}$ and let’s suppose that $[\psi]$ fixes $S$ (after
changing $[\psi]$ with the appropriate power). In this case, for any positive integer $k$, we have
\[
d(S,[\psi]^k(T))=d([\psi]^k(S),[\psi]^k(T))=d(S,T).
\]
As we have fixed $S$ and $T$, this implies that the quantities $d(S,[\psi]^k(T))$ are independent of $k$ and equal.\par
Therefore, by Theorem \ref{thm7.7} we have that there are two iterates $[\psi]^k(T)$, $[\psi]^m(T)$ which
are in the same simplex. If $m>k$, the automorphism $[\psi]^{m-k}$ acts as an isometry on the
centre $X$ of the simplex $\Delta_T=\Delta$ (i.e. $X$ has edges with equal lengths) corresponding to $T$, which means that after possibly
taking a further power, it acts as the identity on the quotient graph corresponding to $\Delta$. It follows that $[\psi]$  virtually fixes any point of $\Delta$ and, in particular, it virtually fixes $T$.\par
As $T$ was an arbitrary point of $M_{[\phi]}$,  any element of $M_{[\phi]}$ is (virtually) fixed by any element of $B'$.
\end{proof}

As an immediate corollary, assuming that the intersection of the centraliser with
the stabilisers of some point which admits a train track representative is finite, we obtain that $C([\phi])$ is "small".

\begin{corollary}\label{col7.10}
     Let $[\phi]\in\mathrm{Out}(G,\mathcal{G})$ be an iwip outer automorhism. If there is some $T\in M_{[\phi]}$ with $C([\phi])\cap \mathrm{Stab}(T)$ finite, then:
     \begin{enumerate}
         \item $C([\phi])$ is a normal $\mathbb{Z}$-extension of a torsion group.
         \item If we further assume that $\mathrm{Out}(G,\mathcal{G})$ is virtually torsion free, then $C([\phi])$ is virtually $\mathbb{Z}$.
     \end{enumerate}
\end{corollary}
\begin{proof}
    The first part is an immediate application of Theorem \ref{thm7.8}, as, under our assumption,
any automorphism of $B'$ must be a torsion element. Lastly,  every torsion subgroup is finite, in a virtually torsion free group, hence the second part follows.
\end{proof}

\begin{remark}\label{gibcent}
We note that there are iwip outer automorphisms whose  (relative) centralisers fail to be virtually cyclic. We fix a free product decomposition $G=G_1\ast\langle a\rangle \ast\langle b\rangle$ where $a,b$ are of infinite order and we denote by $F_2=\langle a\rangle \ast\langle b\rangle$ the “free part” of $G$. Let $\mathcal{O}$ be the corresponding relative outer space, associated to the free product
decomposition above (here the only elliptic free factor is $G_1$). We define the (unique) automorphism $\phi\in\mathrm{Aut}(G,\mathcal{G})$ which is given by the relations
\[
\phi(g)=\begin{cases}
    g,& g\in G_1 \\ bg_1,& g=a \\ ab,& g=b
\end{cases}
\]
for some non-trivial $g_1\in G_1$. It is easy to see
that the corresponding outer automorphism $[\phi]\in\mathrm{Out}(G,\mathcal{G})$ is iwip. Note that every factor automorphism of $G_1$, which fixes $g_1$, commutes with $\phi$.  It
follows that $C([\phi])$  contains a big subgroup of $\mathrm{Out}(G,\mathcal{G})$ isomorphic to a subgroup of the
automorphisms of $G_1$  consisting of those automorphisms fixing $g_1$.
\end{remark}
It is well known that for an iwip automorphism $[\phi]$ of a free group, any element of its centraliser has a common power with $[\phi]$. As before, this does not hold in our case, but we are able to give a description involving stabilisers of points in $\mathcal{O}$.

\begin{corollary}\label{col7.12}
    Let $[\phi]\in\mathrm{Out}(G,\mathcal{G})$ be an iwip outer automorhism. If $[\psi]\in C([\phi])$ is not virtually elliptic and $T\in M_{[\phi]}$, then there are (non-zero) powers $M',N'$ such that $[\phi]^{M'}=[\psi]^{N'}\cdot [\alpha]$, where $[\alpha]\in C([\phi])\cap \mathrm{Stab}(T)$ (where $\mathrm{Stab}(T)$ denotes the stabiliser of
the tree $T$ in $\mathrm{Out}(G,\mathcal{G})$). In particular, if there is some $T\in M_{[\phi]}$ such that $C([\phi])\cap \mathrm{Stab}(T)$ is finite, then $[\phi]$ and $[\psi]$ have common non-zero powers.  
\end{corollary}

\begin{proof}
   We denote by $\sigma$ the stretching map, restricted on $C([\phi])$ and we treat $\sigma(C([\phi]))$ as the additive group of integers $\mathbb{Z}$. As $[\phi]$ and $[\psi]$ are not virtually elliptic, $\sigma([\phi])$ and $\sigma([\psi])$ are non-zero integers. In particular, it follows that there are two (non-zero) integers $M,N$ such that $\sigma([\psi]^M)=\sigma([\phi]^N)$. In that case, $\sigma([\psi]^M[\phi]^{-N})=0$, hence $[\theta]=[\psi]^M[\phi]^{-N}\in \mathrm{ker}(\sigma)$. Theorem \ref{thm7.8} implies that  there is some power $k$ such that $[\theta]^k(T)=T$. As $[\psi]\in C([\phi])$, it follows that $[\psi]^{-kM}[\phi]^{kN}(T)=T$, or, equivalently, $[\psi]^{-kM}[\phi]^{kN}=[\alpha]$, where $[\alpha]\in \mathrm{Stab}(T)$. As a consequence, $[\phi]^{kN}=[\psi]^{kM}[\alpha]$ and the corollary follows by setting $M'=kM$ and $N'=kN$.
\end{proof}

By combining our two main theorems, we are able to prove a dichotomy result for two iwip automorphisms $[\phi],[\psi]\in\mathrm{Out}(G,\mathcal{G})$ of a free product $G$. Intuitively, we show that either $[\phi]$ and $[\psi]$ are "independent", i.e. they have some iterates generating a free subgroup of rank two, or they are "strongly related", i.e. some iterates of them differ by an element which virtually fixes a point of $\mathcal{O}$.

\begin{theorem}\label{main2}
    Let $[\phi],[\psi]$ be two iwip outer automorphisms in $\mathrm{Out}(G,\mathcal{G})$. Then we have the following dichotomy:
    \begin{enumerate}
        \item either $\langle [\phi]^m,[\psi]^n\rangle\cong F_2$ for some $m,n\in\mathbb{N}$, or
        \item $[\phi]^k=[\psi]^l\cdot [\alpha]$ for some $k,l\in\mathbb{N}$ and $[\alpha]\in\mathrm{Out}(G,\mathcal{G})$ such that $[\alpha]$ (virtually) fixes an element of $\mathcal{O}$.
    \end{enumerate}
\end{theorem}

\begin{proof}
   For $[\phi]$, $[\psi]$ and the corresponding laminations $\Lambda^+=\Lambda_{[\phi]}^+$, $\Gamma^+=\Lambda_{[\psi]}^+$, items $1$ and $2$ of Lemma \ref{lem3.4.2} are satisfied since $[\phi]$ and $[\psi]$  are iwip. Thus, by Proposition \ref{prp6.9}, the forward $[\psi]$-iterates of $\Lambda^+$ weakly converge to $\Gamma^+$ unless $\Lambda^+=\Gamma^-$. Similarly for the $[\phi]$-iterates of $\Gamma^+$. Hence, either items $3$ and $4$ of Lemma \ref{lem3.4.2} hold, in which case suitable powers of $[\phi]$ and $[\psi]$ generate a free group, or $\Lambda^+=\Gamma^-$ and $\Lambda^-=\Gamma^+$. In the latter case, $[\psi]\in \mathrm{Stab}(\Lambda^-)$. Then $[\phi]\cdot\mathrm{ker}(\sigma)=x^k\in\mathbb{Z}$ and  $[\psi]\cdot\mathrm{ker}(\sigma)=x^l\in\mathbb{Z}$, so $[\phi]^l\cdot [\psi]^{-k} \in \mathrm{ker}(\sigma)$. From Theorem \ref{thm7.2} we have that $[\phi]^l\cdot [\psi]^{-k}$ virtually fixes a point of $\mathcal{O}$. 
\end{proof}
Lastly, using Remark \ref{gibcent}, we may construct examples of two iwip automorphisms that do not have iterates producing a free subgroup of rank two, and neither have two common powers (these powers differ by an element $[\alpha]$ as above). On the other hand, when we restrict our attention to finite free products we can extend the well known result of the free case (\autocite[Proposition~3.7]{Lam}).

\begin{corollary}\label{maincol}
 Let $[\phi],[\psi]$ be two iwip outer automorphisms of $G=G_1\ast\ldots\ast G_k\ast F_p$, where each factor $G_i$ is finite, that do not
have common powers. Then high powers of $[\phi]$ and $[\psi]$ freely generate $F_2$.    
\end{corollary}
\begin{proof}
    By Corollary \ref{cor7.3} we have that $\mathrm{ker}(\sigma)$ is finite. The proof now is identical to {{\cite[Corollary 2.15]{Lam}}}.
\end{proof}

\printbibliography

\end{document}